\documentclass[11pt]{article}
\usepackage{amsmath}[1996/11/01]
\usepackage{amssymb,amsthm,amsxtra}
\usepackage{amscd}
\usepackage[margin=1in]{geometry}

\def\[#1\]{\begin{equation}#1\end{equation}}
\makeatletter
\def\beq{%
   \relax\ifmmode
      \@badmath
   \else
      \ifvmode
         \nointerlineskip
         \makebox[.6\linewidth]%
      \fi
      $$
   \fi
}
\def\eeq{%
   \relax\ifmmode
      \ifinner
         \@badmath
      \else
         $$
      \fi
   \else
      \@badmath
   \fi
   \ignorespaces
}

\def\enddisplaymath{\eeq\global\@ignoretrue}
\makeatother

\newtheorem{thm}{Theorem}
\newtheorem{cor}[thm]{Corollary}
\newtheorem{lem}[thm]{Lemma}
\newtheorem{prop}[thm]{Proposition}

\theoremstyle{remark}
\newtheorem*{rem}{Remark}
\newtheorem{rems}{Remark}[thm]

\theoremstyle{definition}
\newtheorem{defn}{Definition}

\numberwithin{equation}{section}
\numberwithin{thm}{section}
\numberwithin{eg}{section}

\renewcommand{\P}{\mathbb P}

\DeclareMathOperator{\PGL}{PGL}

\DeclareMathOperator{\Pic}{Pic}

\DeclareMathOperator{\Ext}{Ext}

\DeclareMathOperator{\End}{End}
\DeclareMathOperator{\Hom}{Hom}

\DeclareMathOperator{\Fitt}{Fitt}
\DeclareMathOperator{\rank}{rank}
\DeclareMathOperator{\coker}{coker}
\DeclareMathOperator{\supp}{supp}
\DeclareMathOperator{\tr}{tr}
\DeclareMathOperator{\Tr}{Tr}
\DeclareMathOperator{\Spec}{Spec}
\DeclareMathOperator{\im}{im}

\newcommand{\sO}{\mathcal O}
\newcommand{\sExt}{\mathcal Ext}
\newcommand{\sHom}{\mathcal Hom}

\newcommand{\Spl}{\mathcal Spl}

\newcommand{\Tor}{\mathcal Tor}
\newcommand{\Vect}{\mathcal Vect}

\DeclareMathOperator{\Coh}{Coh}
\newcommand{\dR}{{\bf R}}
\newcommand{\dL}{{\bf L}}
\newcommand{\ratto}{\dashrightarrow}

\begin{document}

\title{Birational morphisms and Poisson moduli spaces}
  \author{
Eric M. Rains\\Department of Mathematics, California
  Institute of Technology}

\date{July 25, 2019}
\maketitle

\begin{abstract}
We study birational morphisms between smooth projective surfaces that
respect a given Poisson structure, with particular attention to induced
birational maps between the (Poisson) moduli spaces of sheaves on those
surfaces.  In particular, to any birational morphism, we associate a
corresponding ``minimal lift'' operation on sheaves of homological
dimension $\le 1$, and study its properties.  In particular, we show that
minimal lift induces a stratification of the moduli space of simple sheaves
on the codomain by open subspaces of the moduli space of simple sheaves
on the domain, compatibly with the induced Poisson structures.

\end{abstract}

\tableofcontents

\section{Introduction}

If $X$ is a smooth projective surface equipped with a Poisson structure,
then the moduli space of sheaves on $X$ inherits a corresponding Poisson
structure (\cite{BottacinF:1995,HurtubiseJC/MarkmanE:2002b}, see below for
precise details and a slight extension), and for any line bundle on $X$,
twisting by that bundle preserves the Poisson structure on the moduli
space.  This fact underlies the construction of the (generalized) Hitchin
integrable systems \cite{HurtubiseJC/MarkmanE:2002a}, which correspond to
those sheaves on $X$ which are the direct images of line bundles on smooth
curves.  If the curve meets the center of the Poisson structure (the curve
where the Poisson structure fails to be symplectic) transversely, we can
blow up a point of intersection, and consider the image of the same line
bundle on the strict transform.  After finitely many such steps, one
obtains a sheaf supported entirely on the open symplectic leaf of the
surface, and the corresponding neighborhood of the moduli space is itself
symplectic.  In particular, one in this way obtains a large group acting as
symplectomorphisms of a symplectic space, a discrete integrable system.

Recent developments \cite{P2Painleve} suggest that there should be an
analogue of this picture for noncommutative surfaces, in which the
integrable system is deformed to a generalized discrete Painlev\'e
equation.  The largest conceptual difficulty here is that direct images of
line bundles are quite rare in the noncommutative picture, so we no longer
have such a natural way to lift sheaves through a birational morphism.  For
that matter, even in the commutative setting, while being a direct image of
a line bundle on a smooth curve is an open condition, the standard
construction gives little insight into what happens outside this open
subspace.

The objective of the present note is to give an alternate description for
this lifting operation that immediately extends to arbitrary sheaves of
homological dimension $\le 1$.  Although this does not quite give a morphism of
moduli spaces (as it fails to preserve flatness of a family of sheaves), we
do obtain a natural stratification of the ``downstairs'' moduli space such
that the lifting operation is a morphism on each stratum, with open image.
Moreover, in the Poisson case, the strata are all Poisson subspaces, and
all morphisms are compatible with the Poisson structure.

The main ingredient in the construction is the observation that there are
two other natural ways to lift a sheaf of homological dimension $\le 1$ through a
birational morphism.  The first is the usual inverse image: although this
is larger than we want in the curve case, it is still a very well behaved
functor.  In particular, sheaves of homological dimension $\le 1$ on a surface
are acyclic for the inverse image, and the direct image of the inverse
image is naturally isomorphic to the original sheaf.  The other natural
lifting operation is slightly less obvious: it turns out that in the case
of a birational morphism of smooth surfaces, the exceptional inverse image
functor, normally only a functor on the derived categories, is itself a
left-derived functor, and the resulting functor $\pi^!$ has the same nice
properties as $\pi^*$.  Adjointness then gives a natural transformation
$\pi^*\to \pi^!$, and the lifting operation we want is the image of this
natural transformation.  (This can also be characterized as the minimal
sheaf on the blown up surface which is acyclic and has the correct direct
image.)

As part of our investigations of this ``minimal lift'' operation, we need
to understand the ``invisible'' sheaves, those with trivial direct image
and higher direct images, as these are precisely the sub- and quotient
sheaves that have no effect on the direct image.  It turns out that the
invisible sheaves form an abelian subcategory of the category of
(coherent) sheaves on the domain, equivalent to the category of
(finitely generated) modules on a certain ring (a finite $k$-algebra).
The kernel and cokernel of the natural map $\pi^*M\to\pi^!M$ are
invisible, and a number of questions about the minimal lift reduce to
questions about these sheaves.  (Conversely, the equivalence of the
category of invisible sheaves to a module category relies on a
construction of projective objects as kernels of maps $\pi^*M\to \pi^!M$.)

The plan of the paper is as follows.  First, we remind the reader of what a
Poisson structure means in the algebraic setting (including in finite
characteristic), and give a classification of (smooth, projective) Poisson
surfaces up to birational equivalence respecting the Poisson structure.  We
find that every non-symplectic Poisson surface is birational to a Poisson
ruled surface $\P({\cal O}_C\oplus \omega_C)$ where $C$ is a smooth curve,
and the center of the Poisson structure is disjoint from the given section
of the ruling.  When $g(C)\ge 1$, these are precisely the surfaces that
arise in generalized Hitchin systems; in the rational case, there are some
additional cases not yet covered in the literature, which we will consider
in more detail in \cite{rat_Hitchin}.  Of course, the rational case is
particularly interesting when it comes to birational maps, since only in
that case do we obtain birational maps that do not preserve the ruling.

Next, we consider some technicalities involving the precise moduli space we
will be using.  Since in general the notion of stability requires a choice
of ample bundle, it can be awkward to compare stability conditions on
either side of a morphism; in addition, semistable points can interact
badly with the Poisson structure.  If we allow the moduli space to become
slightly more complicated, we can relax ``stable'' to the more intrinsic
``simple'', i.e., having no nonscalar endomorphisms.  Although this space
is no longer a scheme, we can still make sense of what it means for it to
be a Poisson structure.  The arguments in the literature mostly carry over
to show that this is Poisson, and covered by smooth symplectic leaves; we
use a reduction of \cite{HurtubiseJC/MarkmanE:2002b} to the case of vector
bundles, and give a new proof in that case that is somewhat simpler and
extends with little difficulty to arbitrary characteristic.

We next consider how the two inverse image functors behave on sheaves of
homological dimension $\le 1$, allowing us to define the minimal lift of
such a sheaf.  We then investigate invisible sheaves, as indicated above.
Following that, we prove our main result on the interaction between the
minimal lift and the Poisson structure, and also investigate how certain
natural ``direct image of twist of minimal lift'' operations interact with
the Poisson structure.  We conclude by considering as an application the
question of rigidity: which sheaves on a Poisson surface are the unique
point of their symplectic leaf?  In the case of sheaves with
$1$-dimensional support, we show that the rigid sheaves are just the line
bundles on $-2$-curves of blowups.  This is again mainly a feature of the
rational surface case, and will be investigated further in
\cite{rat_Hitchin}.

{\bf Acknowledgements}. The author would like to thank P. Etingof,
T. Graber, and A. Okounkov for helpful conversations.  This work was
partially supported by a grant from the National Science Foundation,
DMS-1001645.

\section{Poisson surfaces}

Recall that a Poisson algebra over a commutative ring $R$ is a commutative
$R$-algebra $A$ equipped with an $R$-bilinear map $\{,\!\}:A\times A\to A$
which is a biderivation (a derivation in each variable) and satisfies the
identities
\[
\{f,f\}=0,\qquad
\{f,\{g,h\}\}+\{g,\{h,f\}\}+\{h,\{f,g\}\}=0
\]
for all elements $f$, $g$, $h$.  (Note that we make no assumption about
characteristic here.)  By universality of K\"ahler differentials, to
specify an alternating biderivation, it is equivalent to specify a
homomorphism
\[
\wedge^2\Omega_{A/R}\to A.
\]
of $A$-modules.  Moreover, if $\{,\!\}$ is an alternating biderivation, then
the left-hand side of the Jacobi identity is an alternating triderivation,
which is required to vanish.  A map $\pi:A\to B$ of Poisson $R$-algebras is
called a Poisson map if
\[
\pi(\{f,g\}) = \{\pi(f),\pi(g)\};
\]
we will also need to consider {\em anti-Poisson} maps, such that
\[
\pi(\{f,g\}) = -\{\pi(f),\pi(g)\}.
\]

These notions localize in a natural way, so one may thus define a Poisson
structure on a scheme $X/S$ to be a morphism
\[
\alpha:\wedge^2\Omega_{X/S}\to \sO_X
\]
of $\sO_X$-modules such that the associated map
\[
\wedge^3\Omega_{X/S}\to \sO_X
\]
(defined locally in the obvious way) is 0.  Note that this induces a
Poisson algebra structure on $\Gamma(U,\sO_X)$ for all open subsets $U$.

One can define Poisson and anti-Poisson morphisms of schemes accordingly.
In scheme-theoretic terms, a Poisson morphism $(X,\alpha_X)\to (Y,\alpha_Y)$ is
a morphism $f:X\to Y$ such that the diagram
\[
\begin{CD}
f^*\wedge^2\Omega_{Y/S}@>f^*\alpha_Y>> f^*\sO_Y\\
   @VVV @| \\
\wedge^2\Omega_{X/S}   @>\alpha_X>> \sO_X
\end{CD}
\]
commutes, and similarly, with appropriate sign changes, for anti-Poisson
morphisms.  Note that in general the identity morphism gives an
anti-Poisson morphism $(X,\alpha)\cong (X,-\alpha)$.  (We will also have
occasion to use the obvious analogues for a general biderivation.)

\begin{prop}
  Let $X$ be a Poisson scheme, and suppose $f:X\to Y$ is a morphism such
  that the natural map $f_*\sO_X\to \sO_Y$ is an isomorphism.
  Then there is a unique Poisson structure on $Y$ making $f$ Poisson.
\end{prop}

\begin{proof}
  The composition
\[
f^*\wedge^2\Omega_{Y/S}\to \wedge^2\Omega_{X/S}\to \sO_X
\]
induces by adjointness a map
\[
\wedge^2 \Omega_{Y/S}\to f_*\sO_X\cong \sO_Y,
\]
i.e. an alternating biderivation.  To check that this satisfies the Jacobi
identity, we may reduce to the case $S=\Spec(R)$, $Y=\Spec(A)$ for an
$R$-algebra $A$, in which case this biderivation is just the induced
biderivation on $\Gamma(X,\sO_X)\cong A$.  But this is clearly a Poisson
bracket.
\end{proof}

\begin{rem}
Examples include a projective birational morphism $\pi:X\to Y$ with $X$
integral and $Y$ normal (and locally Noetherian), as well as the morphism
$\rho:X\to C$ associated to a rationally ruled surface.  (Of course, in the
latter case, the induced Poisson structure is trivial!)
\end{rem}

\begin{prop}
  Let $Y$ be a Poisson scheme, and suppose $f:X\to Y$ is a morphism such
  that $f^*\Omega_{Y/S}\to \Omega_{X/S}$ is an isomorphism (e.g. if $f$ is
  \'etale).  Then there is a unique Poisson
  structure on $X$ making $f$ Poisson.
\end{prop}

\begin{proof}
Clearly, the biderivation on $X$ must be
\[
\wedge^2 \Omega_{X/S}\to \wedge^2\pi^*\Omega_{Y/S}\to \pi^*\sO_Y\cong \sO_X,
\]
which is manifestly alternating.  The isomorphism on differentials in
particular implies that a derivation vanishes iff it vanishes on functions
pulled back from $Y$ (i.e., on pullbacks of sections of $Y$ on open subsets
$U$).  In particular, it follows that the Jacobi triderivation vanishes as
required.
\end{proof}

A {\em Poisson subscheme} of a Poisson scheme $(X/S,\alpha)$ is a locally
closed subscheme $Y/S\subset X/S$ which locally satisfies $\{f,{\cal
  I}_Y\}\subset {\cal I}_Y$ where ${\cal I}_Y$ is the ideal sheaf of $Y$
and $f$ is any local section of $\sO_X$.  In particular, we immediately
obtain an induced Poisson structure on $Y$ such that the corresponding
inclusion morphism is Poisson, and such a structure exists iff $Y$ is a
Poisson subscheme.

An important class of Poisson schemes are the symplectic schemes.  A {\em
  symplectic scheme} is a pair $(X,\omega)$ such that $X/S$ is smooth and
$\omega\in \wedge^2\Omega_{X/S}$ is a closed $2$-form that induces an
isomorphism $\Omega_{X/S}\cong T_{X/S}$.  Call $(X,\omega)$ {\em
  presymplectic} if $\omega$ is nondegenerate, but not necessarily closed.
Then $\omega$ induces a biderivation
\[
\alpha:\Omega_{X/S}\otimes \Omega_{X/S}\xrightarrow{1\otimes\omega}
\Omega_{X/S}\otimes T_{X/S}\to \sO_X,
\]
which, as the inverse of an alternating form, is alternating.  Conversely,
if $\alpha$ is an alternating biderivation which is a nonsingular alternating
form on $\Omega_{X/S}$ at every point, then it determines a presymplectic
structure on $X$.

\begin{prop}
Suppose $(X,\omega)$ is a presymplectic scheme.  Then $(X,\omega)$ is
symplectic iff the corresponding biderivation is Poisson.
\end{prop}

\begin{proof}
Indeed, just as in characteristic 0, $\omega$ induces an isomorphism
\[
\wedge^k\Omega_{X/S}
\cong
\wedge^k\Hom(\Omega_{X/S},\sO_X)
\cong
\Hom(\wedge^k\Omega_{X/S},\sO_X)
\]
for any $k$, which for $k=3$ takes $d\omega$ to the triderivation coming from
the Jacobi identity.
\end{proof}
%

\begin{rem}
  In characteristic 0, it is well-known that a Poisson subscheme of a
  symplectic scheme is necessarily open.  (An ideal of a symplectic local
  ring is Poisson iff it is preserved by all derivations, but this is
  impossible in characteristic 0 for a proper, nontrivial ideal.)  This
  fails in finite characteristic; indeed, in characteristic $p$, the
  principal ideal generated by a $p$-th power will always be Poisson.
\end{rem}

By a {\em Poisson surface}, we will mean a pair $(X,\alpha)$ where $X/S$ is a
projective surface with geometrically smooth and irreducible fibers, and
$\alpha$ is a Poisson structure on $X$ over $S$ which is not identically 0 on
any fiber.  Note that since $X$ is a surface,
\[
\wedge^2\Omega_{X/S}\cong \omega_{X/S}
\]
and
\[
\wedge^3\Omega_{X/S}=0,
\]
and thus any section of the anticanonical bundle $\omega_{X/S}^{-1}$
determines a Poisson structure on $X/S$.  Thus a nontrivial Poisson
structure is determined up to a scalar by the divisor of the corresponding
section of $\omega_{X/S}^{-1}$, the {\em associated anticanonical curve},
on the complement of which the surface is symplectic.  The scalar can
itself be identified: if $C_\alpha$ denotes the associated anticanonical
curve, then the short exact sequence (which depends on $\alpha$ only via
$C_\alpha$) 
\[
0\to \omega_X\to \omega_X(C_\alpha)\to \omega_{C_\alpha}\to 0
\]
induces an injection $H^0(\omega_X(C_\alpha))\to H^0(\omega_{C_\alpha})$.
A Poisson structure with associated anticanonical curve $C_\alpha$ is a
nonzero global section of $\omega_X^{-1}(-C)$, so an isomorphism
$\sO_X\cong \omega_X^{-1}(-C)$.  The inverse of this isomorphism determines
a nonzero global section of $\omega_X(C_\alpha)$, and thus a
nonzero holomorphic differential on $C_\alpha$ in the kernel of the
connecting map $H^0(\omega_{C_\alpha})\to H^1(\omega_X)$.  Conversely, any
nonzero element of that kernel determines a nonzero global section
$\omega_X(C_\alpha)$, the inverse of which is a Poisson
structure with associated anticanonical curve $C_\alpha$.  Note that the
differential on $C_\alpha$ scales inversely with the Poisson structure.

As observed above, one can push Poisson structures forward through proper
birational morphisms; in the case of Poisson surfaces, one can be fairly
precise about lifting as well.  Since everything is local on the base, we
will assume $S=\Spec(\bar{k})$ for some algebraically closed field
$\bar{k}$.  Note that if $\pi:X\to Y$ is a birational morphism of smooth
projective surfaces over $\bar{k}$, then there is a (unique) effective
divisor $e_\pi$ supported on the exceptional locus such that $K_X\sim
\pi^*K_Y + e_\pi$.

\begin{lem}
  Let $(Y,\alpha)$ be a Poisson surface over $\bar{k}$, with associated
  anticanonical curve $C$, and suppose $\pi:X\to Y$ is a birational
  morphism.  Then there is a (unique) Poisson structure on $X$ compatible
  with that of $Y$ iff the divisor $\pi^*C-e_\pi$ is effective.
\end{lem}

\begin{proof}
  Indeed, $\pi^*C-e_\pi$ is the unique anticanonical divisor agreeing with
  $C$ where $\pi$ is an isomorphism, so if there is a Poisson structure on
  $X$, its associated anticanonical curve must be $\pi^*C-e_\pi$.  In
  particular, $\pi^*C-e_\pi$ must be effective.  If it is, then it
  determines a Poisson structure up to scalar multiplication, and that
  Poisson structure agrees with $\alpha$ (again up to scalar multiplication)
  where $\pi$ is an isomorphism.  We may thus eliminate the scalar freedom
  and obtain a Poisson structure agreeing with $\alpha$ at the generic point
  of $X$.
\end{proof}

\begin{prop}
  Let $(X,\alpha_X)$, $(Y,\alpha_Y)$ be Poisson surfaces over $\bar{k}$, and
  suppose $f:X\to Y$ is a rational map which is Poisson at the generic
  points of $X$ and $Y$.  Then there exists a Poisson surface $Z$ and
  birational Poisson morphisms $g:Z\to X$, $h:Z\to Y$ such that $f = h
  \circ g^{-1}$.
\end{prop}

\begin{proof}
  Let $f=h\circ g^{-1}$ be the minimal factorization (i.e., $Z$ is the
  minimal desingularization of the graph of $f$); we need to show that $Z$
  has a compatible Poisson structure.  We can factor each of $g$ and $h$ as
  a product of monoidal transformations, and we can arrange that those
  monoidal transformations centered in points of the anticanonical curves
  come first.  We thus reduce to the case that $X$ and $Y$ are isomorphic
  on some neighborhoods of their respective anticanonical curves, and wish
  to show that $f$ is an isomorphism.  Indeed, otherwise, we have an
  identity
\[
g^*C_X-e_g = h^*C_Y-e_h
\]
of divisors on $Z$.  Since $f$ is an isomorphism on neighborhoods of the
anticanonical curves, $g^*C_X=h^*C_Y$, and thus $e_g=e_h$.  If $g$ is not
an isomorphism, then $e_g$ contains some $-1$ curve, which must be
contracted by $h$, contradicting minimality of $Z$.
\end{proof}

We in particular have a well-behaved notion of a Poisson birational map
between Poisson surfaces.  Since our objective in this note is to
understand how moduli spaces of sheaves on Poisson surfaces are affected by
Poisson birational maps, it will be convenient to give a classification of
Poisson surfaces up to Poisson birational equivalence.  (Compare
\cite{BottacinF:1995,BartocciC/MacriE:2005}, which consider minimal Poisson
surfaces, but without considering birational maps between them.)  Call a
Poisson surface {\em standard} if it is isomorphic as an algebraic surface
to the ruled surface $\P(\sO_C\oplus \omega_C)$ for some smooth curve $C$,
in such a way that the morphism $\sO_C\oplus \omega_C\to \omega_C$
determines a section disjoint from the anticanonical curve.

There is also an odd case in characteristic 2.  Let $C$ be a smooth genus 1
curve, and let $X$ be the ruled surface corresponding to the nontrivial
self-extension of $\sO_C$.  In characteristic different from 2, such a
surface has a unique anticanonical curve, but in characteristic 2, there is
actually an anticanonical pencil.  Geometrically, we may take $C$ to be in
Weierstrass form $y^2+a_1xy+a_3y = x^3+a_2x^2+a_4x+a_6$, and the vector
bundle to have transition matrix
\[
\begin{pmatrix} 1 & 0 \\ y/x &
  1\end{pmatrix}
\]
between the complement of the identity and a suitable neighborhood of the
identity.  In general, a Poisson structure on $\P(V)$ corresponds to a
global section of $S^2(V^*)\otimes \det(V)\otimes \omega_C^{-1}$, and thus
in our case an anticanonical curve is cut out by a global section of
$S^2(V^*)$.  This has transition matrix
\[
\begin{pmatrix}
  1 & 0 & y^2/x^2\\
  0 & 1 & y/x\\
  0 & 0 & 1
\end{pmatrix},
\]
a nontrivial extension of $\sO_C^2$ by $\sO_C$.  Since
$\dim\Ext^1(\sO_C,\sO_C)=1$, this has precisely two dimensions of global
sections.  If $z$, $w$ are the standard basis of the trivialization of
$V^*$ away from the identity, then the global sections of $S^2(V^*)$ are
spanned by $w^2$ and $z^2+a_1 z w + x w^2$.  Indeed, if $u$, $v$ are the
basis of the other trivialization, then the latter becomes
\[
u^2 + a_1 u v + (a_2 + a_3 y/x^2 + a_4/x + a_6/x^2) v^2,
\]
the coefficients of which are holomorphic at the identity.  Thus the
generic anticanonical curve on this surface has the form $z^2 + a_1 z + c =
x$ for some $c$.  Note that we can replace $c$ by $c+d^2+a_1 d$ for any
$d$, and thus all such anticanonical curves are related by geometric
automorphisms of $X$.  With that in mind, we call a Poisson surface of this
form ``quasi-standard''.  Note that in general, the degree 2 map from
$C_\alpha$ to $C$ induces a 2-isogeny $J(C_\alpha)\to J(C)$ which is dual
to the Frobenius $2$-isogeny.

\begin{thm}\label{thm:poisson_class}
  Any Poisson surface $(X,\alpha)$ over $\bar{k}$ such that $\alpha$ is not
  an isomorphism is Poisson birational to a standard or quasi-standard
  Poisson surface.
\end{thm}

\begin{proof}
  Since $-K_X$ is nontrivial and effective, $X$ has Kodaira dimension
  $-\infty$, so either $X\cong \P^2$ or $X$ admits a birational morphism to
  a ruled surface $\rho:X'\to C$.  We may thus reduce to the case that $X$
  is a ruled surface; if $X\cong \P^2$, simply blow up a point of the
  anticanonical curve $C_\alpha$.  Since $\rho$ admits infinitely many
  sections, we may choose a section in such a way that the corresponding
  curve is not a component of the anticanonical curve; by mild
  abuse of notation, we call this curve $C$.

  If $C$ intersects $C_\alpha$, perform an elementary transformation
  (blow up a point, then blow down the fiber) based at a point of
  intersection.  Blowing up the point reduces the intersection
  number by 1, and, since $C$ meets the fiber only in the point
  being blown up, blowing down the fiber does not change the
  intersection number.  Thus by induction, there is a sequence of
  Poisson elementary transformations making $C$ disjoint from
  the anticanonical divisor, so we can reduce to the case that $C$
  was already disjoint from $C_\alpha$.

  In particular, $\omega_X\otimes \sO_C\cong \sO_C$, and thus we may use
  adjunction to compute $\omega_C\cong \sO_C(C)$.  Applying
  $\rho_*$ to the short exact sequence
\[
0\to \sO_X\to \sO_X(C)\to \sO_C(C)\to 0
\]
gives the short exact sequence
\[
0\to \sO_C\to \rho_*(\sO_X(C))\to \omega_C\to 0,
\]
where we observe that
\[
X\cong \P(\rho_*(\sO_X(C))).
\]
We claim that this short exact sequence splits unless the surface is
quasi-standard.  If $g(C)=0$, this is immediate.  If $g(C)=1$, then
$-K_X\sim 2C$; since we also have $-K_X\sim C_\alpha$, we conclude that
$h^0(-K_X)>1$, implying that the sequence splits (or the surface is
quasi-standard).  Finally, if $g(C)\ge 2$, then the fact that
$\chi(\sO_{C_\alpha})=0$ and $C_\alpha$ has a degree 2 map to $C$ (note
that $C_\alpha$ cannot have a fiber as a component, since that would
intersect $C$) implies that $C_\alpha$ cannot be integral.  We can thus
write $C_\alpha = C_1+C_2$; since $C_\alpha.f=2$, where $f$ is the
(numerical) class of a fiber, we must have $C_1.f=C_2.f=1$.  But then $C_1$
and $C_2$ are sections disjoint from $C$, giving the desired splitting of
$\rho_*(\sO_X(C))$.  (Such a splitting is moreover unique, so that
$C_\alpha=2C_1$.)
\end{proof}

In particular, we have the following classification of Poisson surfaces, up
to Poisson birational equivalence:
\begin{itemize}
\item[1.] $X$ is the Hirzebruch surface $F_2$, and $C_\alpha$ is a reduced
  anticanonical curve disjoint from the minimal section of $X$.
\item[2.] $X$ is $C\times \P^1$, with $C$ a smooth genus 1 curve, and
  $C_\alpha$ is a union of two disjoint fibers over $\P^1$.
\item[3.] $X$ is $\P(\sO_C\oplus \omega_C)$ for some smooth curve $C$, and
  $C_\alpha=2C_0$, where $C_0$ is the section corresponding to the morphism
  $\sO_C\oplus \omega_C\to \sO_C$.
\item[4.] $X$ is quasi-standard.
\item[5.] $X$ is a minimal surface with trivial anticanonical bundle
  ($K_3$, abelian, certain surfaces in characteristic 2 or 3 with
  nonreduced $\Pic^0$.)
\end{itemize}
Note that the reduced condition is only there to make cases $1$ and $3$
disjoint.  Also, case $1$ splits into subcases, corresponding to the
Kodaira symbols $I_0$, $I_1$, $II$, $I_2$, $III$ (smooth, nodal integral,
cuspidal integral, or two components meeting in a reduced or nonreduced
scheme, respectively).  Cases $2$ and $3$ and the reducible subcases of $1$
have arisen in the theory of generalized Hitchin systems; for a related
interpretation for rational surfaces with integral anticanonical curve, see
\cite{rat_Hitchin}.

Note that when the surface is not rational, either the anticanonical
divisor is trivial and there are no nontrivial Poisson birational maps, or
the (rational) ruling is uniquely determined, and thus any Poisson
birational map between surfaces of the above canonical form is a composition
of elementary transformations.  In contrast, the rational case admits a
rich structure of birational automorphisms respecting the Poisson
structure, again see \cite{rat_Hitchin}.  In particular, the $I_1$ and
$I_2$ subcases are in the same birational equivalence class, as are the
$II$ and $III$ subcases, though there are significant differences in the
corresponding Hitchin-type systems.

We also note that the surfaces birational to quasi-standard Poisson
surfaces can be characterized as those Poisson surfaces with rational
rulings over genus 1 curves such that the anticanonical curve is integral.

\section{Poisson moduli spaces}

In \cite{BottacinF:1995,BottacinF:2000}, Bottacin showed that the moduli
space of stable vector bundles on a Poisson surface in characteristic 0 has
a natural Poisson structure (extending a result of Tyurin
\cite{TyurinAN:1988}, who constructed the form but did not prove the Jacobi
identity), and identified the symplectic leaves of this structure (and in
particular showed that they are algebraic).  This was extended to general
stable sheaves (in fact, simple sheaves), subject to some smoothness
assumptions, in \cite{HurtubiseJC/MarkmanE:2002b}.

Since we wish to understand how these spaces interact with birational maps,
we immediately encounter a problem: the definition of stability depends on
a choice of ample line bundle.  Since there is no canonical way to
transport ample bundles through birational maps, this adds a great deal of
complexity when trying to understand whether a given operation preserves
stability.  One way to avoid this is to replace ``stable'' with ``simple'',
having no nontrivial endomorphisms.  As simplicity is intrinsic, we avoid
any consideration of ample line bundles entirely, yet a result of Altman
and Kleiman means that we still have a reasonably well-behaved moduli space.

Let $X/S$ be locally projective, finitely presented morphism of schemes.  A
{\em family of simple sheaves} on $X$ is a sheaf $M$ on $T\times_S X$, flat
over $T$, such that the fiber $M$ over every geometric point $t\in T$
satisfies $\End(M(t))\cong k(t)$.  We consider two such families $M$, $M'$
equivalent if there is a line bundle ${\cal L}$ on $T$ such that $M'\cong
M\otimes \pi_1^*{\cal L}$; this naturally preserves isomorphism classes of
fibers.  Note that the simplicity condition implies that the line bundle
${\cal L}$ is unique (up to isomorphism) if it exists, since for a simple
family, we have
\[
\Hom_{T\times_S X}(M,M\otimes \pi_1^*{\cal L})\cong \Gamma(T\times_S
X,{\cal L})
\]

Roughly speaking, this moduli problem is represented by an algebraic space.
More precisely, this needs to be extended to the \'etale topology.  Define
a {\em twisted family of simple sheaves} on $X$ parametrized by $T$ to be a
family of simple sheaves parametrized by some \'etale cover $U\to T$ such
that the two induced families on $U\times_T U$ are equivalent.

\begin{thm}\cite{AltmanAB/KleimanSL:1980}
  There is a quasi-separated algebraic space $\Spl_{X/S}$ locally finitely
  presented over $S$ which represents the moduli functor of simple sheaves,
  in the sense that there is a natural bijection between twisted families
  of simple sheaves on $X$ and morphisms to $\Spl_{X/S}$.
\end{thm}

In particular, we obtain such an algebraic space associated to any Poisson
surface, and we wish to show that it is Poisson.  Of course, this
encounters another difficulty: what does it mean for an algebraic space to
be Poisson?  This is less of a difficulty than one might think, since as we
have seen, Poisson structures can be pulled back through \'etale morphisms
(and if a Poisson biderivation descends, the image is clearly Poisson).  In
other words, the notion of a Poisson structure makes perfect sense in the
\'etale topology.  We thus obtain the following definition of a Poisson
structure on an algebraic space.

\begin{defn}
Let ${\cal X}$ be an algebraic space.  A Poisson structure on ${\cal X}$ is
an assignment of a Poisson structure to the domain of every \'etale morphism
$f:U\to {\cal X}$, such that if $g:U'\to U$ is another \'etale morphism,
then $g:(U',\alpha_{f\circ g})\to (U,\alpha_f)$ is a Poisson morphism.
\end{defn}

\begin{rem}
If ${\cal X}$ is a Poisson scheme, then we obtain a Poisson structure in
the above sense by assigning to every \'etale morphism the induced Poisson
structure on its domain.  Thus any Poisson scheme remains Poisson as an
algebraic space.
\end{rem}

Thus to obtain our Poisson moduli space, we need simply define a suitably
canonical Poisson structure on the base of every twisted family of simple
sheaves corresponding to an \'etale morphism to $\Spl_{X/S}$.  In fact, it
will be enough to consider only untwisted families, as the required
compatibility of Poisson structures immediately gives us the conditions
needed to descend the Poisson structure from the \'etale cover of the base
of the twisted family.  Note that a family corresponds to an \'etale
morphism iff it is formally universal.

In \cite{BottacinF:2000}, Bottacin showed that the moduli space of stable
vector bundles on a Poisson surface has a well-behaved foliation by
algebraic symplectic leaves: for any bundle $V$, the subscheme
parametrizing stable vector bundles $V'$ with $V'|_{C_\alpha}\cong
V|_{C_\alpha}$ is a symplectic Poisson subscheme (where $C_\alpha$ is the
anticanonical curve).  This description does not quite carry over in full
generality, although it is straightforward enough to fix: with this in
mind, we define the {\em derived restriction} of a sheaf to be the
complex
\[
M|^{\dL}_{C_\alpha}:=\dL i^*M,
\]
where $i:C_\alpha\to X$ is the inclusion morphism.

We will mainly consider the case of a surface over a separably closed
field, which we suppress from the notation.

\begin{thm}\label{thm:poisson}
  Let $(X,\alpha)$ be a Poisson surface which is not symplectic and not
  birational to a quasi-standard Poisson surface.  Then the moduli space
  $\Spl_X$ has a natural Poisson structure, and for any complex
  $M^\bullet_\alpha$ of locally free sheaves on the associated
  anticanonical curve $C_\alpha$, the subspace of $\Spl_X$ parametrizing
  sheaves $M$ with $M|^{\dL}_{C_\alpha}\cong M^\bullet_\alpha$ is a
  (smooth) symplectic Poisson subspace.
\end{thm}

\begin{rems}
With this in mind, we define a {\em symplectic leaf} of $\Spl_X$ to be
the subspace corresponding to some fixed derived restriction, or more
generally a union of components of such a subspace.
\end{rems}

\begin{rems}
  Of course, this continues to hold when $X$ has symplectic fibers (subject
  to the technical condition that $\Pic^0(X)$ is smooth), with $\Spl_X$
  itself smooth and symplectic.  (See, e.g., Chapter 10 of
  \cite{HuybrechtsD/LehnM:1997}.)  The arguments below encounter
  difficulties in the symplectic case, however.  The claim fails when $X$
  is birational to a quasi-standard surface.
\end{rems}

\begin{rems}
  Since the biderivation can be defined over an arbitrary base, and the
  further conditions to be a Poisson structure are closed on the base, we
  conclude that $\Spl_{X/S}$ is Poisson for any Poisson surface over a
  reduced scheme $S$, such that no fiber of $X$ is symplectic.  The claim
  for the symplectic leaves is harder to generalize, mainly since it is no
  longer true in general that conditions of the form
  $M|^{\dL}_{C_\alpha}\cong M^\cdot_\alpha$ are locally closed.  It appears
  the correct condition to place on a family $M^\cdot_\alpha$ over a
  reduced base is that it can be represented by a two-term perfect complex,
  and that
\[
\dim\Hom(M^\cdot_\alpha,M^\cdot_\alpha)
-
\dim\Ext^{-1}(M^\cdot_\alpha,M^\cdot_\alpha)
\]
is constant.  (By the proof of Proposition \ref{prop:symp_dims} below, this
condition is necessary for the tangent space to the symplectic leaf to have
constant dimension.)  With this proviso, the full Theorem most likely holds
over an arbitrary reduced base, and a careful restatement of the constancy
condition should extend to any locally Noetherian base.  (For $X$ rational
(and probably for ruled surfaces over $C$ with $g(C)=1$), the reduced case
is probably sufficient, since suitable moduli stacks of anticanonical
surfaces are reduced; for $g(C)>1$, though, most components of the moduli
functor have obstructed deformations.)
\end{rems}

In other words, we need to construct a Poisson structure on the base of
every formally universal family of simple sheaves, and show that the
foliation of the base by isomorphism class of derived restrictions to the
anticanonical curve is a foliation by symplectic Poisson subspaces.

The first requirement is to compute the cotangent sheaf.

\begin{lem}
Let $M$ be a formally universal family of simple sheaves on $X$
parametrized by $U$.  Then there is a natural isomorphism
\[
\Omega_U\cong \sExt^1_U(M,M\otimes \omega_X).
\]
\end{lem}

\begin{proof}
Fix an ample line bundle $\sO_X(1)$ on $X$, and for each pair $(m,n)$ of
nonnegative integers, consider the fine moduli space of subsheaves
$V\subset \sO_X(-m)^{\oplus n}$ such that the quotient $M$ is simple and
acyclic for $R\Hom(\sO_X(-m),{-})$, and the induced map
\[
\Hom(\sO_X(-m),\sO_X(-m)^{\oplus n})\to \Hom(\sO_X(-m),M)
\]
is an isomorphism.  On the one hand, this moduli space can be directly
identified as a disjoint union of open subschemes of Quot schemes, so
Theorem 3.1 of \cite{LehnM:1998} describes its cotangent sheaf.  On the
other hand, taking the quotient yields a map to $\Spl_X$ making it a
principal $\PGL_n$-bundle over an open subspace of $\Spl_X$.  We may thus
proceed as in the proof of Theorem 4.2 op.~cit.~to calculate the cotangent
sheaf of that subspace of $\Spl_X$.  Since as $m$, $n$ varies these
subspaces cover $\Spl_X$, the result follows.
\end{proof}

The required biderivation on $U$ is then given (following
\cite{TyurinAN:1988}) by
\[
\Omega_U^{\otimes 2}
\cong
\sExt^1_U(M,M\otimes \omega_X)^{\otimes 2}
\xrightarrow{1\otimes \alpha}
\sExt^1_U(M,M\otimes \omega_X)
\otimes
\sExt^1_U(M,M)
\to
\sO_U,
\]
where the last map is the trace pairing (i.e., Serre duality for
the smooth morphism $X\times U\to U$).  Note the commutative diagram
\[
\begin{CD}
\sExt^1_{U}(M,M\otimes\omega_X)^{\otimes 2}
@>>{1\otimes\alpha}>
\sExt^1_{U}(M,M\otimes\omega_X)
\otimes
\sExt^1_{U}(M,M)
\\
@VV{\alpha\otimes 1}V  @VVV
\\
\sExt^1_{U}(M,M)
\otimes
\sExt^1_{U}(M,M\otimes\omega_X)
@>>> \sO_{U}
\end{CD}
\]

Since the above biderivation is clearly functorial in $U$, and is invariant
under twisting $M$ by any line bundle, so in particular by line bundles
pulled back from $U$, it remains only to show that it is Poisson, and to
verify the claim about symplectic leaves.

Following \cite{HurtubiseJC/MarkmanE:2002b}, there are three cases to
consider, depending on the homological dimension of the simple sheaf.
Simple sheaves of homological dimension 2 are necessarily structure sheaves
of closed points, and thus the corresponding component of
$\Spl_X$ is naturally isomorphic to $X$ itself, and the induced
Poisson structure is the same.  The case of simple sheaves of homological
dimension $<2$ can then be reduced to that of locally free simple sheaves.
This reduction applies equally well in finite characteristic, though we
will need to do some more work (beyond the arguments of
\cite{HurtubiseJC/MarkmanE:2002b}) to identify the symplectic leaves.  We
will consider this case shortly, but will first deal with locally free
sheaves.

\subsection{Locally free sheaves}

The argument of \cite{BottacinF:1995} could most likely be carried over to
finite characteristic, but is sufficiently complicated to make it somewhat
difficult to verify this.  It turns out, however, that Bottacin's
identification of the symplectic leaves can be used to give a simpler proof
which also easily carries over to finite characteristic not equal to 2.
(In fact, one can show that Poissonness in characteristic 0 implies it in
all finite characteristics, which is how we deal with characteristic 2, but
smoothness of the symplectic leaves is more delicate, and in any event it
seems worth giving a simpler argument even for the characteristic 0 case.)

We first need to know that the theorem applies to invertible sheaves.

\begin{prop}\label{prop:Pic_inj}
The natural restriction map
\[
\Pic^0(X)\to \Pic^0(C_\alpha)
\]
is an injective map of group schemes, and the natural Poisson structure is
trivial on $\Pic(X)\subset \Spl_X$.
\end{prop}

\begin{proof}
  It follows from the previous classification that either $X$ is
  rational, or there is a splitting of the induced map from $C_\alpha$
  to the base $C$ of the ruling.  If $X$ is rational, then
  $\Pic^0(X)=1$, and there is nothing to prove, while if $C_\alpha\to
  C$ splits, then the composition
\[
\Pic^0(X)\to \Pic^0(C_\alpha)\to \Pic^0(C)
\]
agrees with the natural isomorphism.

Taking derivatives gives an injection
\[
H^1(\sO_X)\to H^1(\sO_{C_\alpha})
\]
implying that
\[
H^1(\alpha):H^1(\omega_X)\to H^1(\sO_X)
\]
is 0.  But this implies for invertible $M$ that the induced map
\[
\Ext^1(M,M\otimes\omega_X)\to \Ext^1(M,M)
\]
is 0, and thus so is the Poisson structure.
\end{proof}

\begin{rem}
  Note that this fails when $X$ is birational to a quasi-standard Poisson
  surface.  Indeed, in that case, the composition
  \[
  \Pic^0(C)\cong \Pic^0(X)\to \Pic^0(C_\alpha)
  \]
  is the Frobenius isogeny, and thus has nonreduced kernel ($\mu_2$
  or $\alpha_2$, depending on whether $J(C)$ is ordinary or
  supersingular).  In particular, the fibers of restriction to
  $C_\alpha$ can fail to be smooth, and thus Theorem \ref{thm:poisson}
  fails in this case.  Moreover, since the map $H^1(\omega_X)\to
  H^1(\sO_X)$ induced by $\alpha$ is an isomorphism, we find that
  the natural biderivation on the $1$-dimensional scheme $\Pic^0(X)$
  induces a nondegenerate pairing on fibers of the cotangent bundle.
  In particular, the natural biderivation does not give a Poisson
  structure, as it is not even alternating.  Note that we can exclude
  both this case and the symplectic cases on which the conclusion
  of Theorem \ref{thm:poisson} fails by insisting that the kernel
  of $\Pic^0(X)\to \Pic^0(C_\alpha)$ be reduced.
\end{rem}

\begin{cor}
If $X$ is rationally ruled over the smooth curve $C$, then
$h^1(\sO_{C_\alpha})=g(C)+1$.
\end{cor}

\begin{proof}
Indeed, one has a short exact sequence
\[
0\to H^1(\sO_X)\to H^1(\sO_{C_\alpha})\to H^2(\omega_X)\to 0
\label{ses:cor3.5}
\]
where the first map is injective by the proposition.  Since $h^1(\sO_X)=g(C)$,
$h^2(\omega_X)=1$, the claim follows.
\end{proof}

\begin{rem}
  Again, in the quasi-standard case, this fails with $h^1(\sO_{C_\alpha})=1$.
\end{rem}

Now, let $\Vect_X$ denote the subspace of $\Spl_X$ classifying simple
locally free sheaves; a sheaf being locally free is an open condition, so
$\Vect_X$ is open in $\Spl_X$.

\begin{lem}
The algebraic space $\Vect_X$ is smooth.
\end{lem}

\begin{proof}
It suffices to prove that the moduli problem is formally smooth;
i.e., any infinitesimal deformation of a vector bundle can be extended.
But the obstructions to deformations of $V$ are classified by
\[
\Ext^2(V,V)\cong \Hom(V,V\otimes\omega_X)^*\cong \Hom(V,V(-C_\alpha))^*.
\]
Since by assumption $\Hom(V,V)$ is generated by the identity, 
the injective map
\[
\Hom(V,V(-C_\alpha))\to \Hom(V,V)
\]
must be 0.
\end{proof}

In particular, since $\Vect_X$ is smooth, its tangent sheaf is
well-behaved, and duality gives
\[
{\cal T}_U\cong \sExt_U(V,V)
\]
where $V$ is any formally universal family with base $U$.

\begin{lem}
  For any locally free sheaf $V_\alpha$ on $C_\alpha$, the subspace of
  $\Vect_X$ on which $V\otimes \sO_{C_\alpha}\cong V_\alpha$ is locally closed
  and smooth, and the corresponding conormal sheaf is the radical of the
  natural bilinear form on $\Omega_{\Vect_X}$.
\end{lem}

\begin{proof}
  That the subspace is locally closed is standard.
%
%
For smoothness, we need
  to show that any obstruction to a deformation inside this space must
  vanish.  As in \cite{BottacinF:1995,BottacinF:2000}, we obtain a
  self-dual long exact sequence the middle of which is
\[
\to
 \Hom(V_\alpha,V_\alpha)
\to \Ext^1(V,V\otimes\omega_X)\to \Ext^1(V,V)\to \Ext^1(V_\alpha,V_\alpha)\to
 \Ext^2(V,V\otimes \omega_X)\to
\label{eq:le_vect}
\]
Since $\Vect_X$ is smooth, any infinitesimal deformation of $V$ can be
extended, and the question is whether the extension can be chosen in such a
way that $V_\alpha$ remains constant.  Equivalently, any such extension
determines (as an extension of the trivial deformation of $V_\alpha$) a
class in $\Ext^1(V_\alpha,V_\alpha)$, which can be made trivial iff it is
in the image of $\Ext^1(V,V)$, or equivalently iff its image in
$\Ext^2(V,V\otimes\omega_X)$ is trivial.  Since $V$ is simple, the trace
map
\[
\tr:\Ext^2(V,V\otimes\omega_X)\to H^2(\omega_X)
\]
is an isomorphism (it is dual to the natural map $k\to \End(V)$).  It is
thus equivalent to check whether the trace of the class in
$\Ext^1(V_\alpha,V_\alpha)$ is in the image of $\tr(\Ext^1(V,V))$.  But by
\cite{ArtamkinIV:1988}, the trace map is the map induced by $V\mapsto
\det(V)$ on the corresponding spaces of deformations.  In particular the
trace of the given class in $\Ext^1(V_\alpha,V_\alpha)$ is the class of the
associated deformation of
\[
\det(V_\alpha)\cong \det(V)|_{C_\alpha}.
\]
We thus reduce to the case that $V_\alpha$, $V$ are invertible sheaves, in which
case it follows by Proposition \ref{prop:Pic_inj} that the deformation of
$V$ must be trivial, so certainly extends.

Now, by smoothness, $U$ is an \'etale neighborhood in the given locally closed
subspace iff we have an exact sequence
\[
0\to {\cal T}_U\to \sExt^1_U(V,V)\to \sExt^1_U(V_\alpha,V_\alpha),
\]
where the first map is the Kodaira-Spencer map and the second map comes
from the relative analogue of the above long exact sequence.  Equivalently,
the dual sequence
\[
\sHom_U(V_\alpha,V_\alpha)\to \sExt^1_U(V,V\otimes\omega_X)\to \Omega_U\to 0
\]
must be exact, and thus the conormal sheaf is equal to the radical, as
required.
\end{proof}

We now assume the characteristic is not 2.  To check that the bracket is
Poisson, it will suffice to verify this for the induced bracket on
$\Omega_U$ with $U$ an \'etale neighborhood in the subspace parametrizing
locally free sheaves with $V|_{C_\alpha}\cong V_\alpha$.  (Indeed, it then
follows that the ideal sheaf measuring the failure to be Poisson is
contained in the ideal sheaf of the symplectic leaf, and the intersection
over all symplectic leaves is the trivial ideal sheaf.)  The claim about
the radical shows that the corresponding map $\Omega_U\to {\cal T}_U$ is an
isomorphism, and thus the bivector induces a 2-form which is closed iff the
bivector is Poisson.

Restricting to an open affine subscheme of $U$ as convenient, we may assume
that $V|_{C_\alpha}$ is isomorphic to the pullback of $V_\alpha$
(rather than merely isomorphic up to twisting by a line bundle on $U$).  If we
then choose an open affine covering $U_i$ of $X$, we find that the
corresponding transition functions $g_{ij}$ are constant on $C_{\alpha}$, and
thus the natural \v{C}ech cocycle (representing the relevant portion of the
Atiyah class of $V$, see \cite{HuybrechtsD/LehnM:1997} for a nice
exposition)
\[
\beta_{ij}:=(d_U g_{ij}) g_{ij}^{-1} \in \sExt^1_U(V,V\otimes\Omega_U)
\]
is trivial on $C_\alpha$, so may be viewed as a cocycle in
\[
\sExt^1_U(V,V(-C_\alpha)\otimes \Omega_U).
\]
The $2$-form associated to the above pairing is then induced by the cocycle
\[
\gamma_{ijk}:=\frac{1}{2}\Tr(\beta_{ij}\wedge g_{ij}\beta_{jk}g_{ij}^{-1})
\in
Z^2(\sO_X(-C_\alpha)\otimes \Omega^2_U)
\]
by applying $\alpha$ and then taking the cohomological trace.  (This of
course fails in characteristic 2.)  Both maps are constant on $U$, so to
show that the $2$-form is closed, it suffices to show that
$d_U\gamma_{ijk}$ is a coboundary.  Since
\begin{align}
  g_{ij}\beta_{jk}g_{ij}^{-1} &= \beta_{ik}-\beta_{ij},\\
  d_U\beta_{ij} &= -\beta_{ij}\wedge\beta_{ij},\\
  \Tr(\beta_{ij}\wedge\beta_{ij}) &= 0,
\end{align}
we can easily compute
\begin{align}
d_U\gamma_{ijk}
&=
\frac{1}{2}d_U\Tr(\beta_{ij}\wedge\beta_{ik})\\
&=
-\frac{1}{2}\Tr(\beta_{ij}\wedge\beta_{ij}\wedge\beta_{ik})
+\frac{1}{2}\Tr(\beta_{ij}\wedge\beta_{ik}\wedge\beta_{ik})\\
&=
-\frac{1}{6}\check{d} \Tr(\beta_{ij}\wedge\beta_{ij}\wedge\beta_{ij}),
\end{align}
and thus $d_U\gamma_{ijk}$ is actually a (\v{C}ech) coboundary of a cochain
that vanishes to third order on $C_\alpha$.  It follows in particular that
$d_U\gamma=0$ in $\sExt^2_U(V,V(-C_\alpha))\otimes\Omega^3_U$,
as required.

This finishes our proof of the theorem for locally free sheaves, except in
characteristic 2 and 3.  For those cases, we observe that we can extend the
Poisson surface to one over a mixed characteristic discrete valuation ring
(this fails for the bad characteristic 2 case!), and since $\Ext^2(V,V)=0$,
we can similarly lift $V$ and the lift will remain simple.  The pairing on
the cotangent bundle is defined on this larger family and on the generic
fiber is alternating and satisfies the Jacobi identity.  But since these
are closed conditions, they also hold on the special fiber.

\begin{rem}
  Note that for a map $A:R^n\to V^n$ with $R$ a commutative ring and $V$ an
  $R$-module, we have
\[
\Tr(A\wedge A\wedge A)
=
\sum_{i,j,k} A_{ij}\wedge A_{jk}\wedge A_{ki}
=
3\sum_{i<j,k} A_{ij}\wedge A_{jk}\wedge A_{ki}
\in \wedge^3(V).
\]
This lets us define $\Tr(A\wedge A\wedge A)/6$ even when $3$ is not
invertible, allowing the above argument to work directly in that case.
\end{rem}

\subsection{Sheaves of homological dimension $\le 1$}
\label{sec:moduli_hd1}

Now let $U\to \Spl_X$ be an \'etale neighborhood with corresponding family
$M$ of simple sheaves, and suppose that the fibers of M have homological
dimension $\le 1$.  Let $\pi$ denote the projection $U\times X\to
U$.

Twisting $M$ by a line bundle has no effect on the biderivation, so we
may assume that the natural map $\pi^*\pi_*M\to M$ is surjective, and both
$M$ and $M\otimes \omega_X$ are $\pi_*$-acyclic.  Then, as observed in
\cite{HurtubiseJC/MarkmanE:2002b}, one obtains a corresponding locally free
resolution
\[
0\to V\to W\to M\to 0
\]
with $W=\pi^*\pi_*M$ and $V$ simple, such that we can reconstruct $M$ from
$V$.  Indeed, acyclicity implies $\pi_*M$ is locally free, and the
homological dimension condition then ensures that $V$ is locally
free.  Moreover, the induced map
\[
\sHom_U(W,\sO_X)\to \sHom_U(V,\sO_X)
\]
is an isomorphism, and thus given $V$ we can recover $M$ as the cokernel of
the natural injection
\[
V\to \sHom_U(\sHom_U(V,\sO_X),\sO_X).
\]
It follows in particular that $V$ is also a family of simple sheaves,
giving an injective morphism $U\to \Vect_X$, and we need to compare the two
induced biderivations, then understand the symplectic leaves.

For the first part, we use a mild variant of the calculation in
\cite{HurtubiseJC/MarkmanE:2002b}.  The functoriality of long exact
sequences gives a commutative diagram
\[
\begin{CD}
\sHom_U(V,M)  @>>> \sExt^1_U(M,M)\\
@VVV                   @VVV\\
\sExt^1_U(V,V)@>>> \sExt^2_U(M,V)
\end{CD}
\]
Here, the top arrow is surjective, since
\[
\sExt^1_U(W,M)\cong \sHom(\pi_*(M),R^1\pi_*(M))=0,
\]
and the right arrow is injective since
\[
\sExt^1_U(M,W)\cong \sHom(R^1\pi_*(M\otimes\omega_X),\pi_*M)=0.
\]
There is also a natural map $\delta:\sExt^1_U(M,M)\to \sExt^1_U(V,V)$
coming from the above construction (the derivative of the map of moduli
spaces).  Since the top arrow above is surjective, we may represent any
self-extension of $M$ as the pushforward of the resolution $0\to V\to W\to
M\to 0$ by some morphism $\phi:V\to M$.  This gives a self-extension $M'$
as the cokernel in the short exact sequence
\[
0\to V\to W\oplus M\to M'\to 0.
\]
The surjection $W\to M$ then induces a surjection $W\oplus W\to M'$,
so an exact sequence
\[
0\to V'\to W\oplus W\to M'\to 0,
\]
where is a self-extension of $V$ (the image of $M'$ under $\delta$).
Pulling this back through the injection $M\to M'$ gives a short exact
sequence
\[
0\to V'\to V\oplus W\to M\to 0,
\]
so that $V'$ is the pullback through $\phi$ of $0\to V\to W\to M\to 0$.

In other words, given a class in $\sExt^1_U(M,M)$, we can compute its image
under $\delta$ by choosing a preimage in $\sHom_U(V,M)$ and mapping that to
$\sExt^1_U(V,V)$.  In particular, we can put the differential into the
diagram without breaking commutativity.  Note in particular that
commutativity implies that the differential is injective, and in fact
identifies $\sExt^1_U(M,M)$ with the image of the map $\sHom_U(V,M)\to
\sExt^1_U(V,V)$.  We similarly have a commutative square
\[
\begin{CD}
\sHom_U(V,M\otimes\omega_X)  @>>> \sExt^1_U(M,M\otimes\omega_X)\\
@VVV                   @VVV\\
\sExt^1_U(V,V\otimes\omega_X)@>>> \sExt^2_U(M,V\otimes\omega_X)
\end{CD}
\]
Since the connecting maps in this diagram are the duals (up to sign) of the
maps in the original diagram, we see that the dual of the differential also
fits into this diagram; more precisely, we must take the negative of the
dual in order to account for the signs.

Without the diagonal maps, the two squares fit into a commutative cube
induced by $\alpha:\omega_X\to \sO_X$.  We claim that this diagram remains
commutative when we introduce the diagonal maps.  In other words, we need
to show that the composition
\[
\sExt^1_U(V,V\otimes\omega_X)\xrightarrow{-\delta^*}
\sExt^1_U(M,M\otimes\omega_X)\xrightarrow{\alpha}
\sExt^1_U(M,M)\xrightarrow{\delta}
\sExt^1_U(V,V)
\]
agrees with the map induced by $\alpha$.  If we compose both maps with the
connecting map $\sExt^1_U(V,V)\to\sExt^2_U(M,V)$, they agree.  If we can show
that the image of $\sExt^1_U(V,V\otimes\omega_X)\to \sExt^1_U(V,V)$ is
contained in the image of $\delta$, we will be done, since then the error
is in the image of $\delta$, which injects in $\sExt^2_U(M,V)$.  The image
of $\delta$ is the same as the image of $\Hom_U(V,M)$, and thus is the same
as the kernel of the map $\sExt^1_U(V,V)\to \sExt^1_U(V,W)$.  By
commutativity of long exact sequences, it will thus suffice to show that
the map $\sExt^1_U(V,W\otimes\omega_X)\to \sExt^1_U(V,W)$ is 0.  We have a
commutative square
\[
\begin{CD}
\sExt^1_U(W,W\otimes\omega_X)@>\alpha >> \sExt^1_U(W,W)\\
       @VVV                              @VVV\\
\sExt^1_U(V,W\otimes\omega_X)@>\alpha >> \sExt^1_U(V,W)
\end{CD}
\]
The left arrow is an isomorphism, by duality from the fact that
$R^1\pi_*(V)\cong R^1\pi_*(W)$, while the top arrow is $0$ by Proposition
\ref{prop:Pic_inj}, since $\pi_*\omega_X\to \pi_*\sO_X$ is the 0 morphism.
It follows that the bottom arrow is 0 as required.

We have thus shown that the above map of moduli spaces identifies the two
biderivations up to sign, and thus since the biderivation on $\Vect_X$ is
Poisson, so is the biderivation on $U$.  (The minus sign above makes this
an anti-Poisson morphism.)  There are two things remaining for us to do:
first, we need to show that any locally closed subscheme of $U$ on which
$M|^{\dL}_{C_\alpha}$ is fixed is a Poisson subscheme, and second that it
is symplectic.  Of course, we will show both at once if we can identify the
putative leaves of $U$ with the leaves of $\Vect_X$.

Now, $M|^{\dL}_{C_\alpha}$ is represented by the complex $V_\alpha\to
W_\alpha$, where $V_\alpha$ and $W_\alpha$ are the respective restrictions to
$C_\alpha$.  We claim that, just as $V$ determines $M$, $V_\alpha$ determines
$M|^{\dL}_{C_\alpha}$.  Consider the commutative diagram
\[
\begin{CD}
0@>>>\Hom(W,\sO_X)@>>>\Hom(W_\alpha,\sO_{C_\alpha})@>>>\Ext^1(W,\omega_X)@>>>\Ext^1(W,\sO_X)\\
@. @VVV             @VVV                @VVV               @VVV  \\
0@>>>\Hom(V,\sO_X)@>>>\Hom(V_\alpha,\sO_{C_\alpha})@>>>\Ext^1(V,\omega_X)@>>>\Ext^1(V,\sO_X)
\end{CD}
\]
coming from the exact sequence
\[
0\to \omega_X\xrightarrow{\alpha} \sO_X\to \sO_{C_\alpha}\to 0.
\]
Each vertical map fits into a long exact sequence; since $M$ and $M\otimes
\omega_X$ are acyclic, we may apply Serre duality to find that the first
and third vertical maps are isomorphisms, while the fourth map is
injective.  We thus find that the given map $V_\alpha\to W_\alpha$ induces an
isomorphism
\[
\Hom(W_\alpha,\sO_{C_\alpha})\cong \Hom(V_\alpha,\sO_{C_\alpha}).
\]
In particular, $\Hom(V_\alpha,\sO_{C_\alpha})$ is a free
$\End(\sO_{C_\alpha})$-module, and the map $V_\alpha\to W_\alpha$ can be
identified with the natural map
\[
V_\alpha\to \Hom_{\End(\sO_{C_\alpha})}(\Hom(V_\alpha,\sO_{C_\alpha}),\sO_{C_\alpha}).
\]
It follows that the symplectic leaves of $\Vect_X$ pull back to closed
subspaces of the putative symplectic leaves of $U$.

Conversely, we can recover $V_\alpha$ from $M|^{\dL}_{C_\alpha}$ and the
numerical invariants of $M$.  (This also implies that the condition
$M|^{\dL}_{C_\alpha}\cong M^\cdot_\alpha$ is locally closed.)  Since $M$
and $M\otimes \omega_X$ are acyclic, we have a four-term exact sequence
\[
0\to H^{-1}(M|^{\dL}_{C_\alpha})\to H^0(M\otimes \omega_X)\to H^0(M)\to
H^0(M|^{\dL}_{C_\alpha})\to 0,
\]
and thus (locally on the base) we may write
\[
W_\alpha \cong H^0(M|^{\dL}_{C_\alpha})\otimes \sO_{C_\alpha}\oplus
\sO_{C_\alpha}^l
\]
where
\[
l = h^0(M|^{\dL}_{C_\alpha})-h^0(M) = h^0(M|^{\dL}_{C_\alpha})-\chi(M).
\]
Moreover, this isomorphism identifies the map $W_\alpha\to
M|^{\dL}_{C_\alpha}$ with the direct sum of the natural morphism
\[
H^0(M|^{\dL}_{C_\alpha})\otimes \sO_{C_\alpha}\to M|^{\dL}_{C_\alpha}
\]
and the zero morphism
\[
\sO_{C_\alpha}^l\to M|^{\dL}_{C_\alpha}.
\]
But then we can recover $V_\alpha$ as the ``kernel'' of this morphism; more
precisely, the mapping cone of this morphism is a complex representing the
sheaf $V_\alpha$.  In particular, we have
\[
V_\alpha \cong \sO_{C_\alpha}^l\oplus V_{\min},
\]
where
\[
V_{\min}\to H^0(M|^{\dL}_{C_\alpha})\otimes \sO_{C_\alpha}
\]
is the canonical complex quasi-isomorphic to $M|^{\dL}_{C_\alpha}$.  (Of
course, when $\Tor_1(M,\sO_{C_\alpha})=0$, so $M|^{\dL}_{C_\alpha}$ is a
sheaf, $V_{\min}$ is an actual kernel, of the natural (surjective!)
morphism $H^0(M|_{C_\alpha})\otimes \sO_{C_\alpha}\to M|_{C_\alpha}$.)

We thus conclude that the putative symplectic leaf of $\Spl_X$ containing
$M$ can be identified with a Poisson subspace of a symplectic leaf of
$\Vect$.  It remains only to show that this Poisson subspace is open.
(This is automatic in characteristic 0, since then any Poisson subspace of a
symplectic space is open.  The reader only interested in characteristic 0
may still find the direct argument instructive, however.)

Fix a simple sheaf $M$ satisfying the above conditions, with corresponding
locally free sheaf $V$.  In the symplectic leaf of $\Vect_X$ corresponding
to $V_\alpha$, we first impose the following open conditions on the fibers
$V'$:
\begin{itemize}
\item[1.] $V'$ has the same numerical Chern class as $V$.  (This is both
  open and closed.)
\item[2.] $H^0(V')=H^2(V')=0$.  (This is open by semicontinuity)
\item[3.] $\Ext^2(V',\sO_X)=0$ and $\dim\Hom(V',\sO_X)\le
  \rank(W)$. (Also semicontinuity).
\end{itemize}
If $U'$ is an \'etale neighborhood in this open subspace, then
$\sExt^1_{U'}(V',\omega_X)$ is locally free, since the other $\Ext$ groups
are trivial, and thus it has the same rank as
\[
\Ext^1(V,\omega_X)\cong \Ext^1(W,\omega_X)\cong H^1(W)^*
\]
Since $W$ is trivial, we conclude
\[
\rank(\sExt^1_{U'}(V',\omega_X)) = g\rank(W),
\]
where $g=\dim H^1(\sO_X)$ is the genus of the curve over which $X$ is
rationally ruled.  Also, since $V_\alpha$, $W_\alpha$ are fixed and
$h^0(\sO_{C_\alpha})=g+1$, we may compute
\[
h^0(V_\alpha^*) = h^0(W_\alpha^*) = (g+1)\rank(W).
\]
But then the long exact sequence
\[
0\to \sHom_{U'}(V',\sO_X)\to \sHom_{U'}(V',\sO_{C_\alpha})
 \to \sExt^1_{U'}(V',\omega_X)\to\cdots
\]
implies that $\sHom_{U'}(V',\sO_X)$ contains a locally free sheaf of rank
at least $\rank(W)$.  Since we have already imposed an upper bound on the
rank, we conclude that $\sHom_{U'}(V',\sO_X)$ is locally free of rank
$\rank(W)$.  We thus recover an injective map
\[
V'\to \sHom_{U'}(\sHom_{U'}(V',\sO_X),\sO_X).
\]
If we let $M'$ be the cokernel of this map, then it certainly is generated
by global sections, and we can further impose the open conditions that $M'$
and $M'\otimes \omega_X$ are $\pi_*$-acyclic.  We thus obtain a
neighborhood of $V$ in its symplectic leaf such that every sheaf in that
neighborhood comes from a sheaf $M'$.

This completes the proof of Theorem \ref{thm:poisson}.
\qed

\section{The minimal lift}

Much of what we have to say about the minimal lift construction is
independent of the Poisson structure, so we will for the moment
drop the assumption that our surfaces are Poisson.  They also involve
only local considerations, so we drop the requirement that the
surfaces be projective, but also restrict to geometric fibers.  Thus
for the purposes of Sections 4-6, an ``algebraic surface'' will mean
an irreducible smooth $2$-dimensional scheme over an algebraically
closed field.  In addition, when we say that $\pi:X\to Y$ is a
birational morphism, we will mean that $X$ and $Y$ are algebraic
surfaces and $\pi$ is projective.

Given a birational morphism $\pi:X\to Y$, there is an obvious way
to transport coherent sheaves on $Y$ to coherent sheaves on $X$,
namely the inverse image functor $\pi^*$.  This is very well-behaved
on sheaves of homological dimension $\le 1$.

\begin{lem}
  Let $\pi:X\to Y$ be a birational morphism.  Then any coherent
  sheaf $M$ on $Y$ of homological dimension $\le 1$ is $\pi^*$-acyclic.
  Moreover, the sheaf $\pi^*M$ is $\pi_*$-acyclic, and the natural
  map $M\to \pi_*\pi^*M$ is an isomorphism.
\end{lem}

\begin{proof}
If $\pi$ is monoidal and $M=\sO_Y$, this is a standard fact about
birational morphisms of surfaces.  The result for general $\pi$ follows by
induction along a factorization of $\pi$ into monoidal transformations, and
the case $M$ locally free follows easily.

In general, choose a locally free resolution
\[
0\to V\xrightarrow{B} W\to M\to 0.
\]
Applying $\pi^*$ gives an exact sequence
\[
\pi^*V\xrightarrow{\pi^*B} \pi^*W\to \pi^*M\to 0,
\]
with $\pi^*V$ and $\pi^*W$ locally free.  Since $B$ is injective, $\pi^*B$
is still injective on the generic fiber; since $\pi^*V$ is locally free, we
conclude that $\pi^*B$ is injective, and thus $M$ is $\pi^*$-acyclic as
required.  The remaining claim follows upon taking the direct image of the
resulting short exact sequence.
\end{proof}

Regarding sheaves of homological dimension $\le 1$, we have the following
characterization.

\begin{prop}
A coherent sheaf $M$ on an algebraic surface $X$ has homological
dimension $\le 1$ iff it has no subsheaf of the form $\sO_p$ for
$p$ a closed point of $X$.
\end{prop}

\begin{proof}
  We note that $M$ has homological dimension $\le 1$ iff
  $\Tor_2(M,\sO_{\overline{x}})=0$ for all (not necessarily closed) points
  $x\in X$.  If $\dim(\overline{x})=2$, then $\sO_{\overline{x}}=\sO_X$, so
  there is no condition, while if $\dim(\overline{x})=1$, then
  $\overline{x}$ is a divisor on $X$, and its structure sheaf has
  homological dimension $1$.  In other words, we find that $M$ has
  homological dimension $\le 1$ iff $\Tor_2(M,\sO_p)=0$ for all closed
  points $p$.  Now, the minimal resolution of $\sO_p$ in the local ring at
  $p$ is self-dual, and thus we have an isomorphism
\[
\Tor_2(M,\sO_p)\cong \sHom_X(\sO_p,M).
\]
Since
$\sHom_X(\sO_p,M)$ is supported on $p$, $\sHom_X(\sO_p,M)=0$ iff
$\Hom(\sO_p,M)=0$, iff $\sO_p$ is not a subsheaf of $M$.
\end{proof}

Since $\pi_*\pi^*M=0$, if $\pi_*$ had a right adjoint $\pi^!$, then
the isomorphism $\pi_*\pi^*M\cong M$ would induce morphisms $\pi^*M\to
\pi^!M$ and $M\to \pi_*\pi^!M$.  Of course, since $\pi_*$ is not
right exact, it cannot have a right adjoint, but duality theory
gives an adjoint to the derived functor.  In our case, we actually
obtain something stronger.

\begin{lem}
Let $\pi:X\to Y$ be a birational morphism of algebraic surfaces.
Then there exists a right exact functor $\pi^!:\Coh_Y\to \Coh_X$,
acyclic on locally free sheaves, such that for any coherent sheaves
$M$, $N$ on $X$ and $Y$ respectively (or bounded complexes of such
sheaves),
\[
\dR\sHom(\dR\pi_*M,N) \cong \dR\sHom_Y(M,\dL\pi^!N).
\]
Moreover, there is a natural isomorphism
\[
N\cong \dR\pi_*\dL\pi^!N.
\]
\end{lem}

\begin{proof}
Since $\pi_*$ is proper, $\pi_*=\pi_!$, and thus $\dR\pi_*$ has a
right adjoint $\pi^!$ on the derived category.  Since $X$ and $Y$ are smooth,
\[
\omega_X
\cong
(\phi\circ \pi)^!\sO_{\bar{k}}[-2]
\cong
\pi^!(\phi^!\sO_{\bar{k}}[-2])
\cong
\pi^!\omega_Y,
\]
where $\phi:Y\to \Spec(\sO_{\bar{k}})$ is the structure morphism
of the $\bar{k}$-scheme $Y$.  But then for any locally free sheaf $V$ on $Y$,
\[
\pi^!V
\cong
\pi^!(\omega_Y\otimes(\omega_Y^{-1}\otimes V))
\cong
\pi^!\omega_Y\otimes \dL\pi^*(\omega_Y^{-1}\otimes V)
\cong
\pi^*V\otimes \omega_X\otimes \pi^*\omega_Y^{-1}
\]
and thus in particular $\pi^!V$ is a sheaf.  Using a locally free
resolution, we conclude that
\[
\pi^!N^\cdot \cong \dL\pi^*N^\cdot \otimes \omega_X\otimes \pi^*\omega_Y^{-1}
\]
for any complex $N^\cdot$.  In other words, the derived functor $\pi^!$
is the left derived functor of the right exact functor
\[
N\mapsto \pi^*N\otimes \omega_X\otimes \pi^*\omega_Y^{-1}.
\]

For the remaining claim, we have
\[
\dR\pi_*\dL\pi^!N
\cong
\dR\sHom_Y(\sO_X,\dL\pi^!N)
\cong
\dR\sHom(\dR\pi_*\sO_X,N)
\cong
\dR\sHom(\sO_Y,N)
\cong
N.
\]
\end{proof}

\begin{rem}
More generally, if $\pi:X\to Y$ is a proper morphism of Gorenstein
varieties, then
\[
\pi^!N\cong \dL\pi^*N\otimes \omega_X\otimes
\pi^*\omega_Y^{-1} [\dim(X)-\dim(Y)];
\]
this is essentially an observation
of Deligne \cite[Prop.~7]{DeligneP:1966}.
Similarly, acyclicity on sheaves of homological dimension 1 and the
isomorphisms $\dR\pi_*\dL\pi^*N\cong N$, $\dR\pi_*\dL\pi^!N\cong N$ hold
for any proper birational morphism between smooth varieties.  The
next lemma encapsulates the key feature of the surface case, which
is much less common for birational morphisms in higher dimension.
\end{rem}

\begin{lem}
Let $\pi:X\to Y$ be a birational morphism of algebraic surfaces.
Then for any sheaf $M$ on $X$, we have $R^i\pi_*M=0$ for $i\ge 2$.
In particular, any quotient of a $\pi_*$-acyclic sheaf is $\pi_*$-acyclic.
\end{lem}

\begin{proof}
Indeed, the fibers of $\pi$ are 1-dimensional.
\end{proof}

Since $\pi^!$ is a twist of $\pi^*$, we find that sheaves of
homological dimension $1$ are also acyclic for $\pi^!$.  Thus if 
$M$ is a sheaf of homological dimension $1$, we also have a natural
isomorphism
\[
M\cong \pi_*\pi^!M
\]
Either this or the isomorphism $\pi_*\pi^*M\cong M$ induces by adjointness
a natural morphism
\[
\pi^*M\to \pi^!M,
\]
which will be an isomorphism outside the exceptional locus of $\pi$.

\begin{defn}
Let $\pi:X\to Y$ be a birational morphism of algebraic surfaces,
and let $M$ be a coherent sheaf on $Y$ with homological dimension
$\le 1$.  Then the {\em minimal lift} of $M$ is the image $\pi^{*!}M$
of the natural map $\pi^*M\to \pi^!M$.
\end{defn}

\begin{rem}
In the theory of local systems, there is a functor ``middle
extension'', the image of the natural transformation $j_!\to j_*$ where $j$
is an open immersion \cite{KatzNM:1990}.  Since Hitchin systems are
relaxations of moduli spaces of differential equations, it is likely that
this is more than just a formal analogy.  Also, the middle extension is
most naturally defined on perverse sheaves; is there an analogous notion in
the present setting?
\end{rem}

To justify the name, we have the following quasi-universal property.

\begin{prop}
Let $N$ be a sheaf of homological dimension $\le 1$ on $Y$, and
suppose $M$ is a $\pi_*$-acyclic sheaf on $X$ with direct image $N$.
Then $M$ has a natural subquotient isomorphic to $\pi^{*!}N$.
\end{prop}

\begin{proof}
The isomorphism $\pi_*M\cong N$ and its inverse induce by adjointness
natural maps
\[
\pi^*N\to M\to \pi^!N.
\]
It remains only to show that the composition agrees with the canonical
map $\pi^*N\to \pi^!N$.  Since we obtained the factors by adjointness,
it follows that the composition
\[
N\to \pi_*\pi^*N\to \pi_*M\to \pi_*\pi^!N\to N
\]
is the identity, which implies that the map $\pi^*N\to \pi^!N$ is adjoint
to the inverse of the canonical morphism $N\to \pi_*\pi^*N$, as required.
\end{proof}

Of course, we also want to know that $\pi^{*!}M$ itself satisfies this!

\begin{prop}
Let $M$ be a sheaf of homological dimension $\le 1$ on $Y$.  Then
$\pi^{*!}M$ is $\pi_*$-acyclic, and there is a natural
isomorphism $\pi_*\pi^{*!}M\cong M$.
\end{prop}

\begin{proof}
Since $\pi^{*!}M$ is a quotient of the $\pi_*$-acyclic sheaf $\pi^*M$,
it is certainly $\pi_*$-acyclic.  Moreover, the composition
\[
M\cong\pi_*\pi^*M\to \pi_*\pi^{*!}M\to \pi_*\pi^!M\cong M
\]
is the identity, and the map
\[
\pi_*\pi^{*!}M\to \pi_*\pi^!M
\]
is injective.  The claim follows.
\end{proof}

The minimal lift behaves well under composition of birational morphisms.

\begin{prop}
Let $\pi_1:X\to Y$, $\pi_2:Y\to Z$ be birational morphisms of
algebraic surfaces.  Then for any coherent sheaf $M$ on $Z$ with
homological dimension $\le 1$,
\[
(\pi_2\circ\pi_1)^{*!}M
\cong
\pi_1^{*!}\pi_2^{*!}M.
\]
\end{prop}

\begin{proof}
Consider the composition
\[
\pi_1^*\pi_2^*M\to
\pi_1^*\pi_2^{*!}M\to
\pi_1^{*!}\pi_2^{*!}M\to
\pi_1^!\pi_2^{*!}M\to
\pi_1^!\pi_2^!M.
\]
The second map is surjective and the third map injective by definition
of $\pi_1^{*!}$, and the first map is surjective since $\pi_1^*$ is
right exact.  It remains only to show that the fourth map is injective,
for which it will suffice to show that $\pi_2^!M/\pi_2^{*!}M$ has 
homological dimension 1.  If not, then the quotient contains the structure
sheaf of a point, and thus has nontrivial direct image under $\pi_2$, 
contradicting the fact that $\pi_2^{*!}$ and $\pi_2^!M$ are
$\pi_{2*}$-acyclic sheaves with isomorphic direct images.

It follows that $\pi_1^{*!}\pi_2^{*!}M$ is the image of the natural map
\[
(\pi_2\circ\pi_1)^*M\cong \pi_1^*\pi_2^*M\to \pi_1^!\pi_2^!M
\cong
(\pi_2\circ\pi_1)^!M
\]
as required.
\end{proof}

We also want to compare the minimal lift to the lifting operation in the
theory of Hitchin systems.

\begin{prop}
Let $\pi:X\to Y$ be a birational morphism of algebraic surfaces,
and let $C\subset X$ be a curve intersecting the exceptional locus
of $\pi$ transversely.  Then for any torsion-free coherent sheaf
$M$ on $C$, viewed as a sheaf on $X$,
\[
\pi^{*!}\pi_*M\cong M.
\]
\end{prop}

\begin{proof}
The transversality assumption implies that $\pi|_C:C\to Y$ has
$0$-dimensional fibers, so $M$ is certainly $\pi_*$-acyclic.
Moreover, if we twist $M$ by the inverse of a sufficiently ample
bundle, we can arrange that $M$ has no global sections, so that the
same applies to its direct image.  But then $\pi_*M$ certainly
cannot have a 0-dimensional subsheaf, so that $\pi_*M$ has homological
dimension $\le 1$.

It follows that $M$ has $\pi^{*!}\pi_*M$ as a subquotient.
Moreover, they are isomorphic away from the exceptional locus, so
the residual sub- and quotient sheaves of $M$ are supported there.
Transversality then makes both sheaves $0$-dimensional, so that the
torsion-free hypothesis makes the subsheaf trivial.  We thus have
a short exact sequence
\[
0\to \pi^{*!}\pi_*M\to M\to T\to 0
\]
where $T$ is $0$-dimensional.  As before, we compute $\pi_*T=0$ from
the long exact sequence, and thus $T=0$ as required.
\end{proof}

\begin{rem}
In particular, if the sheaf $N$ on $Y$ is an invertible
sheaf on a smooth curve (or simply the image of such a curve under
a morphism), then $\pi^{*!}N$ is the corresponding sheaf on the
strict transform.
\end{rem}

Another important example of minimal lifts is the following.

\begin{prop}
  Let $\pi:X\to Y$ be a birational morphism of algebraic surfaces, and
  suppose that $C$ is a curve on $Y$.  Then $\pi^{*!}\sO_C\cong \sO_{C'}$
  for some curve $C'$.  If $X,Y,\pi$ are Poisson and $C$ is the anticanonical
  curve on $Y$, then $C'$ is the anticanonical curve on $X$.
\end{prop}

\begin{proof}
It suffices to consider the case that $\pi$ is monoidal, with center $p$.
If $p\notin C$, then $\pi^*\sO_C\cong \pi^!\sO_C\cong \sO_{\pi^{-1}C}$,
so the result is immediate.  If $p\in C$, then over the local ring at $p$,
$\sO_C$ has a minimal resolution
\[
0\to (\sO_Y)_p\to (\sO_Y)_p\to \sO_C\otimes (\sO_Y)_p\to 0,
\]
from which we can compute
\[
\pi^{*!}\sO_C\cong \sO_{\pi^*C-e}.
\]

In particular, in the Poisson case, we must have $p\in C_\alpha$, and
$\pi^*C_\alpha-e$ is indeed the anticanonical divisor on $X$.
\end{proof}


\section{Invisible sheaves}

Since $\pi^{*!}M$ is naturally a quotient of $\pi^*M$ and a subsheaf of
$\pi^!M$, and adjointness makes the latter sheaves easy to deal with, this
suggests that an investigation of the corresponding kernel and cokernel
might be rewarding.

A key property of those sheaves is that they are in a sense invisible to
$\pi_*$.  To be precise, as we will see, they satisfy the following
definition.  We again restrict our attention to smooth surfaces over
algebraically closed fields.

\begin{defn}
  Let $\pi:X\to Y$ be a birational morphism of algebraic surfaces.  A
  coherent sheaf $E$ on $X$ is {\em $\pi$-invisible} if it is acyclic
  with trivial direct image.
\end{defn}

\begin{rem}
In other words, $E$ is $\pi$-invisible iff $\dR\pi_*E=0$.
\end{rem}

We will omit $\pi$ from the notation when it is clear from context.
Note that an invisible sheaf is necessarily supported (set-theoretically)
on the exceptional locus, since $\pi$ is an isomorphism elsewhere.
Also, if $E$ had a $0$-dimensional subsheaf, that subsheaf would have
nontrivial direct image; we thus find that invisible sheaves have
homological dimension $1$.

\begin{prop}
Suppose $M^\bullet$ is a complex of sheaves on $X$ such that
$\dR\pi_*M^\bullet=0$.  Then every homology sheaf of $M$ is invisible.
\end{prop}

\begin{proof}
Since $\pi$ has $\le 1$-dimensional fibers, $R^p\pi_*=0$ for $p\ge 2$.
Thus the hypercohomology spectral sequence
\[
R^p\pi_*(H^q(M^\bullet))\Rightarrow R^{p+q}\pi_*M^\bullet
\]
collapses at the $E_2$ page.  Since the limit of the spectral sequence is
0, every term on the $E_2$ page is 0, and thus
\[
\pi_*(H^q(M^\bullet)) = R^1\pi_*(H^q(M^\bullet))=0.
\]
In other words, $H^q(M^\bullet)$ is invisible for all $q$.
\end{proof}

The most important special case of this for our purposes is the following.

\begin{cor}\label{cor:invisible_ker_coker}
  Suppose $f:M\to N$ is a morphism of $\pi_*$-acyclic sheaves on $X$ such
  that $\pi_*f$ is an isomorphism.  Then $\ker(f)$ and $\coker(f)$ are
  invisible.  Conversely, if $f$ has invisible kernel and cokernel,
  then $\pi_*f$ is an isomorphism; moreover, $\im(f)$ is also
  $\pi_*$-acyclic with $\pi_*M\cong \pi_*\im(f)\cong \pi_*N$.
\end{cor}

\begin{proof}
The first claim is immediate from the proposition; for the second,
factoring $f$ through its image reduces to the cases that $f$ is injective
or surjective.  But since the cokernel/kernel is invisible, we obtain an
isomorphism between the remaining two derived direct images, and composing
find that $\pi_*f$ is an isomorphism as required.
\end{proof}

In particular, this implies that the sheaves $\ker(\pi^*M\to \pi^{*!}M)$ and
$\pi^!M/\pi^{*!}M$ are invisible, as we indicated above.

\begin{prop}
The category of $\pi$-invisible sheaves is closed under taking kernels,
cokernels and extensions.  In particular, it is an
abelian category.
\end{prop}

\begin{proof}
  The claim for kernels and cokernels follows from Corollary
  \ref{cor:invisible_ker_coker}.  For extensions
\[
0\to E\to E'\to E''\to 0,
\]
the long exact sequence corresponding to $\dR \pi_*$ immediately tells us
that if two of the sheaves are invisible, then so is the third.
\end{proof}

Another source of invisible sheaves is the following.

\begin{prop}\label{prop:invisible_image}
Suppose $f:M\to N$ is a morphism of sheaves on $X$ such that
$M$ is $\pi_*$-acyclic, and $\pi_*N$ has homological dimension $\le 1$.
If $f$ vanishes outside the exceptional locus of $\pi$, then
$\im(f)$ is invisible, and $\pi_*f=0$.
\end{prop}

\begin{proof}
Since $M$ is $\pi_*$-acyclic, so is its quotient $\im(f)$.  Since $f$
vanishes outside the exceptional locus, $\pi_*\im(f)$ is 0-dimensional; but
$\pi_*N$ has no $0$-dimensional subsheaf.
\end{proof}

Invisible sheaves also interact nicely with the lifting operations.

\begin{prop}
Let $E$ be an invisible sheaf on $X$, and $M$ a sheaf on $Y$ with
homological dimension $\le 1$.  Then
\[
\dR\Hom(\pi^*M,E)=\dR\Hom(E,\pi^!M)=0.
\]
\end{prop}

\begin{proof}
This follows immediately from adjointness:
\begin{align}
\dR\Hom(\pi^*M,E)&\cong \dR\Hom(M,\dR\pi_*E)=0,\notag\\
\dR\Hom(E,\pi^!M)&\cong \dR\Hom(\dR\pi_*E,M)=0.\notag
\end{align}
\end{proof}

\begin{prop}
Let $E$ and $M$ be as before.  Then
\[
\Hom(\pi^{*!}M,E)=\Hom(E,\pi^{*!}M)=0.
\]
\end{prop}

\begin{proof}
Indeed,
\begin{align}
\Hom(E,\pi^{*!}M)&\subset \Hom(E,\pi^!M)=0,\notag\\
\Hom(\pi^{*!}M,E)&\subset \Hom(\pi^*M,E)=0.\notag
\end{align}
\end{proof}

\begin{rem}
We will see below that this actually characterizes those sheaves which are
minimal lifts.
\end{rem}

Invisible sheaves behave well under lifts and direct images.

\begin{lem}
  Suppose $\pi:X\to Y$ and $\phi:Y\to Z$ are birational morphisms of
  algebraic surfaces.  If a sheaf $E$ on $Y$ is $\phi$-invisible, then
  $\pi^*E$, $\pi^!E$, and $\pi^{*!}E$ are $\phi\circ\pi$-invisible.  If a
  sheaf $E$ on $X$ is $\phi\circ\pi$-invisible, then it is
  $\pi_*$-acyclic and $\pi_*E$ is $\phi$-invisible.  Moreover, $\pi_*$
  and $\pi^*$ take projective objects (of the category of invisible
  sheaves) to projective objects, and $\pi_*$ and $\pi^!$ take injective
  objects to injective objects.
\end{lem}

\begin{proof}
First suppose $E$ is $\phi_*$-invisible, and let $E'$ be any of the three
lifts of $E$ to $X$.  Then we have a natural isomorphism
$\dR\pi_*E'\cong E$, and thus
\[
\dR(\phi\circ\pi)_*E'\cong \dR\phi_*\dR\pi_*E'\cong \dR\phi_*E=0
\]
as required.

Now, let $E$ be a $\phi\circ\pi$-invisible sheaf on $X$.  We again have
\[
\dR\phi_*\dR\pi_*E = 0.
\]
Since $\phi$ and $\pi$ have $\le 1$-dimensional fibers, the corresponding
spectral sequence collapses at the $E_2$ page, and thus
\[
R^p\phi_*R^q\pi_*E=0
\]
for $p,q\in \{0,1\}$.  In particular, $\pi_*E$ is invisible, and we need
only show that $E$ is $\pi_*$-acyclic.  But since $R^1\pi_*E$ is supported on
the indeterminacy locus of $\pi^{-1}$, the only way it can have trivial
direct image is to be 0.

The claims about projective and injective objects follow from adjointness.
For instance if $E$ is projective in the category of
$\pi\circ\phi$-invisible sheaves, and $E'$ is any $\phi$-invisible
sheaf, then
\[
\Ext^{p+1}(\pi_*E,E')
\cong
\Ext^{p+1}(E,\pi^!E')
=
0.
\]
\end{proof}

Thus the atomic case (monoidal transformations blowing up a single point)
will be particularly useful.  Here, we can completely characterize the
invisible sheaves.

\begin{lem}
Suppose that $\pi:X\to Y$ is a monoidal transformation, blowing up the
point $p\in Y$ to the exceptional line $e\subset X$.  Then the functor
\[
E\mapsto \Hom(\sO_e(-1),E)
\]
establishes an equivalence between the category of $\pi$-invisible
sheaves and the category of finite-dimensional vector spaces over $\bar{k}$.
\end{lem}

\begin{proof}
  Let $D$ be an effective divisor on $Y$ that has multiplicity $1$ at $p$.
  Then we obtain a short exact sequence
\[
0\to \sO_e(-1)\to \pi^*\sO_D\to \pi^{*!}\sO_D\to 0.
\]
(Use the minimal resolution of $\sO_D$ over the local ring at $p$ to
compute $\pi^{*!}\sO_D$.)  We thus find
\[
R^p\Hom(\sO_e(-1),E)\cong R^{p+1}\Hom(\pi^{*!}\sO_D,E)
\cong R^{p+1}\Hom(\sO_{\tilde{D}},E),
\]
where $\tilde{D}$ is the strict transform of $D$.  Since $\tilde{D}$ is
transverse to the exceptional line, it is transverse to the support of $E$,
and thus the only nonvanishing $\Ext$ group is in degree 1.  We thus
conclude that for any invisible sheaf,
\[
R^p\Hom(\sO_e(-1),E)=0
\]
for $p>0$; in other words, $\sO_e(-1)$ is a projective object in the
category of invisible sheaves.

Since $\sO_e(-1)$ is coherent, with endomorphism ring $k$, it remains only
to show that it generates the category, since then the desired equivalence
follows by Morita theory.  In other words, we need to show that
if $\Hom(\sO_e(-1),E)=0$, then $E=0$.  Since $\tilde{D}$ is transverse to
the exceptional locus, we have a short exact sequence
\[
0\to E\to E(\tilde{D})\to T\to 0
\]
for some $0$-dimensional sheaf $T$.  If $T=0$, then $E$ has support
disjoint from $\tilde{D}$; if nontrivial, it would have $0$-dimensional
support, and thus nontrivial direct image.  Thus $T\ne 0$, and we conclude
that
\[
\chi(E(\tilde{D}))=\chi(T)>0,
\]
so that $E(\tilde{D})$ has global sections, and its direct
image is thus a nontrivial $0$-dimensional sheaf.  It follows from the next
lemma that
\[
\Hom(\sO_e(-1),E)
\cong
\Hom(\sO_e,E(\tilde{D}))\ne 0,
\]
as required.
\end{proof}

\begin{lem}
  Let $\pi:X\to Y$ be the blowup in the closed point $p$ of $Y$, with
  exceptional line $e$, and let $M$ be a coherent sheaf on $X$ of
  homological dimension $\le 1$.  If $M$ is not $\pi_*$-acyclic, then
  $\Hom(M,\sO_e(-2))\ne 0$, while if $\pi_*M$ has homological dimension
  $2$, then $\Hom(\sO_e,M)\ne 0$.
\end{lem}

\begin{proof}
  Since $\pi$ is an isomorphism away from $e$ and $p$, the only way that
  $\pi_*M$ can have homological dimension 2 is if it has a subsheaf
  isomorphic to $\sO_p$.  We can compute
\[
\dR\Hom(\sO_p,\dR\pi_*M)
\cong
\dR\Hom(\dL\pi^*\sO_p,M)
\]
so in particular
\[
\Hom(\sO_p,\pi_*M) \cong \Hom(\pi^*\sO_p,M).
\]
A simple calculation in the local ring at $p$ gives
\[
L_1\pi^*\sO_p\cong \sO_e(-1),
\quad\pi^*\sO_p\cong \sO_e,
\]
giving
\[
\Hom(\sO_p,\pi_*M) \cong \Hom(\sO_e,M),
\]
implying the desired result.

Similarly, if $M$ is not $\pi_*$-acyclic, then $R^1\pi_*M$ is a nontrivial
sheaf supported at $p$, so has a morphism to $\sO_p$.  We compute
\[
\dR\Hom(\dR\pi_*M,\sO_p)
\cong
\dR\Hom(M,\dL\pi^!\sO_p)
\]
which in degree $-1$ gives
\[
\Hom(R^1\pi_*M,\sO_p)
\cong
\Hom(M,L_1\pi^!\sO_p)
\cong
\Hom(M,\sO_e(-2)).
\]
\end{proof}

\begin{cor}
  Let $\pi:X\to Y$ be the blowup in a single point $p\in Y$, with
  exceptional line $e$, and let $M$ be a sheaf of homological dimension $
  \le 1$ on $X$.  If $\Hom(\sO_e(-1),M)=\Hom(M,\sO_e(-1))=0$, then $M\cong
  \pi^{*!}\pi_*M$.
\end{cor}

\begin{proof}
  We first observe that $M$ is $\pi_*$-acyclic with $\pi_*M$ of homological
  dimension $\le 1$.  Indeed, we would otherwise have a nonzero map
  $\sO_e\to M$ or $M\to \sO_e(-2)$, which we could compose with any nonzero
  map $\sO_e(-1)\to \sO_e$ or $\sO_e(-2)\to \sO_e(-1)$.  The second
  composition is necessarily nonzero by injectivity of the second map; the
  first composition could only be zero if the original map had
  $0$-dimensional image.

In particular, we find that we have natural maps
\[
\pi^*\pi_*M\to M\to \pi^!\pi_*M.
\]
The first map is surjective, since its cokernel is invisible, and the
hypotheses imply that $M$ has no nonzero maps to invisible sheaves.
Similarly, the second map is injective since it has invisible kernel.  In
other words, $M$ is the image of the natural map $\pi^*\pi_*M\to
\pi^!\pi_*M$ as required.
\end{proof}

For a more general birational morphism, the exceptional locus is still a
union of finitely many smooth rational curves, which we call the {\em
  exceptional components} of $\pi$.

\begin{lem}
  Let $\pi:X\to Y$ be a birational morphism of algebraic surfaces, and let
  $M$ be a nonzero coherent sheaf on $X$ such that $\pi_*M=0$.  Then there
  is some exceptional component $f$ such that $\Hom(M,\sO_f(-1))\ne 0$.
\end{lem}

\begin{proof}
  If $\pi$ is the blowup in a point $p\in Y$, then either $R^1\pi_*M\ne 0$,
  in which case it maps to $\sO_e(-2)$ and thus $\sO_e(-1)$, or
  $R^1\pi_*M=0$, in which case it is invisible, and thus is isomorphic to
  $\sO_e(-1)^r$.

Otherwise, we can factor $\pi=\pi_1\circ \pi_2$ such that $\pi_1$ blows up
a single point.  If $\Hom(M,\sO_e(-1))\ne 0$, where $e$ is the exceptional
locus of $\pi_1$, then we are done.  Otherwise, $M$ is $\pi_{1*}$-acyclic,
and $\pi_{2*}(\pi_{1*}M)\cong \pi_*M=0$.  Moreover, we have a short exact
sequence of the form
\[
0\to \sO_e(-1)^r\to \pi^*_1\pi_{1*}M\to M\to 0.
\]
Since $\Hom(\sO_e(-1),\sO_f(-1))$ for $f\ne e$, it will suffice to show
that
\[
\Hom(\pi^*_1\pi_{1*}M,\sO_f(-1))\ne 0
\]
for some exceptional component $f\ne e$. But
\[
\Hom(\pi^*_1\pi_{1*}M,\sO_f(-1))
\cong
\Hom(\pi_{1*}M,\pi_{1*}\sO_f(-1))
\cong
\Hom(\pi_{1*}M,\sO_{\pi_1(f)}(-1))
\]
Every exceptional component of $\pi_2$ is the image of some $f\ne e$, and
thus at least one of the latter groups is nonzero.
\end{proof}

\begin{cor}
  Any invisible sheaf has a quotient of the form $\sO_f(-1)$ for some
  exceptional component $f$.
\end{cor}

\begin{proof}
  By the lemma, any invisible sheaf has a nonzero morphism to some
  $\sO_f(-1)$.  The image of such a morphism has the form the form
  $\sO_f(-d)$ for some $d>1$.  Since it is a quotient of an invisible,
  thus $\pi_*$-acyclic, sheaf, it must be $\pi_*$-acyclic.  But then
\[
H^1(\sO_f(-d)) = H^1(\pi_*\sO_f(-d)) = 0,
\]
since it has $0$-dimensional direct image.  It follows that $d=1$.
\end{proof}

\begin{cor}
The category of invisible sheaves is artinian.
\end{cor}

\begin{proof}
Let $E_1\supsetneq E_2\supsetneq E_3\supsetneq\cdots$ be a descending
chain of invisible sheaves.  Each quotient $E_i/E_{i+1}$ is
a nontrivial invisible sheaf, so surjects on some $\sO_f(-1)$.
In particular, relative to any very ample bundle on $X$, $\deg(c_1(E_i))$
is a strictly decreasing sequence of nonnegative integers.
\end{proof}

\begin{cor}
  Any invisible sheaf admits a filtration in which the successive
  quotients are all of the form $\sO_f(-1)$ with $f$ an exceptional
  component.
\end{cor}

\begin{proof}
Any nontrivial invisible sheaf has a nontrivial map to some $\sO_f(-1)$,
and the kernel is invisible.  Iterating gives a descending chain
of invisible sheaves, which must eventually reach 0.
\end{proof}

This gives us an alternate characterization of minimal lifts.

\begin{thm}\label{thm:minimal_if_no_exceptional}
  Let $\pi:X\to Y$ be a birational morphism of algebraic surfaces, and
  suppose that $M$ is a sheaf on $X$ of homological dimension $\le 1$.  If
  $\Hom(M,\sO_f(-1))=0$ for all exceptional components $f$, then $M$ is
  $\pi_*$-acyclic and $\pi$-globally generated.  If $\Hom(\sO_f(-1),M)=0$
  for all exceptional components $f$, then $\pi_*M$ has homological
  dimension $\le 1$.  Moreover, $M$ is a minimal lift iff
  $\Hom(\sO_f(-1),M)=\Hom(M,\sO_f(-1))=0$ for all exceptional components $f$.
\end{thm}

\begin{proof}
  Suppose that $\Hom(M,\sO_f(-1))=0$ for all exceptional components, and
  consider the natural map
\[
\pi^*\pi_*M\to M.
\]
The direct image of this map is an isomorphism, and thus the usual spectral
sequence shows that the cokernel has trivial direct image.  It follows that
if the cokernel is nonzero, then it maps to $\sO_f(-1)$ for some $f$,
giving a corresponding map on $M$.  We thus conclude that $M$ is
$\pi$-globally generated, and thus $\pi_*$-acyclic (as a quotient of the
$\pi_*$-acyclic sheaf $\pi^*\pi_*M$).

Next, note that if $\Hom(\sO_f(-1),M)=0$ for all exceptional components,
then $\Hom(E,M)=0$ for all invisible sheaves $E$.  So it will suffice to
show that if $\pi_*M$ has homological dimension 2, then $\Hom(E,M)\ne 0$
for some invisible $E$.  Factor $\pi=\pi_1\circ \pi_2$ with $\pi_1$
monoidal.  If $\pi_{1*}M$ has homological dimension 2, then
$\Hom(\sO_e,M)\ne 0$, and thus $\Hom(\sO_e(-1),M)\ne 0$.  Otherwise, we
have $\Hom(E,\pi_{1*}M)\ne 0$ for some $\pi_2$-invisible $E$, and thus
$\Hom(\pi_1^*E,M)\ne 0$.  Since $\pi_1^*E$ is $\pi$-invisible, the claim
follows.

Now, suppose that both conditions hold.  Then in the composition
\[
\pi^*\pi_*M\to M\to \pi^!\pi_*M
\]
we have already shown the first map to be surjective, and the second map
has invisible kernel, so must be injective.  That $M\cong
\pi^{*!}\pi_*M$ follows.
\end{proof}

Given an exceptional component $f$, let $f^\vee$ denote the linear
combination of exceptional components such that $f^\vee\cdot
g=-\delta_{fg}$ for exceptional components $g$.  (Note that although $X$ is
not assumed projective, it still has a well-behaved intersection theory, at
least where exceptional divisors are concerned.)  Equivalently,
\[
\sO_g(f^\vee)\cong \sO_g(-\delta_{fg}).
\]
By induction on a factorization of $\pi$ into monoidal transformations, we
find that $f^\vee$ is effective.  Define a sheaf $P_f$ by the exact sequence
\[
0\to P_f\to \pi^*\pi_*\sO_X(-f^\vee)\to \sO_X(-f^\vee).
\]

\begin{lem}
The sheaves $P_f$ are projective objects in the category of
$\pi$-invisible sheaves.  More precisely, $P_f$ is the projective cover
of $\sO_f(-1)$.
\end{lem}

\begin{proof}
Since
\[
\Hom(\sO_X(-f^\vee),\sO_g(-1)) \cong H^0(\sO_g(-1-\delta_{fg}))=0,
\]
it follows that $\sO_X(-f^\vee)$ is $\pi_*$-acyclic and $\pi$-globally
generated, so that the exact sequence defining $P_f$ extends to a short
exact sequence.  In addition, we have
\[
\pi_*\sO_X(-f^\vee)\subset \pi_*\sO_X\cong \sO_Y,
\]
and subsheaves of locally free sheaves have homological dimension $\le 1$.

Thus $P_f$ is an invisible sheaf, and we have
\[
R^p\Hom(P_f,\sO_g(-1))
\cong
R^{p+1}\Hom(\sO_X(-f^\vee),\sO_g(-1))
\cong
H^{p+1}(\sO_g(-1-\delta_{fg}))
\]
In particular, $R^p\Hom(P_f,\sO_g(-1))=0$ for $p>0$; since every
invisible sheaf is an extension of sheaves $\sO_g(-1)$, it follows that
$P_f$ is projective.  Moreover,
\[
\dim H^1(\sO_g(-1-\delta_{fg})) = \delta_{fg},
\]
and thus $P_f$ has a unique map to $\sO_f(-1)$, and maps to no other
component; it is thus the projective cover of $\sO_f(-1)$.
\end{proof}

\begin{rem}
Taking the direct image of the short exact sequence
\[
0\to \sO_X(-f^\vee)\to \sO_X\to \sO_{f^\vee}\to 0
\]
gives
\[
0\to \pi_*\sO_X(-f^\vee)\to \sO_Y\to \pi_*\sO_{f^\vee}\to 0,
\]
so that $\pi_*\sO_X(-f^\vee)$ is the ideal sheaf of the $0$-dimensional
subscheme $\pi(f^\vee)\subset Y$, and
\[
P_f\cong L_1\pi^*\pi_*\sO_{f^\vee}\cong L_1\pi^*\sO_{\pi(f^\vee)}.
\]
\end{rem}

Since the category of $\pi$-invisible sheaves has finitely many
irreducible objects, and every irreducible object has a projective cover,
it is isomorphic to a module category.  More precisely, we have the
following, by a straightforward application of Morita theory.

\begin{thm}
  There is a natural equivalence between the category of $\pi$-invisible
  sheaves and the category of finitely generated modules over the
  finite-dimensional $k$-algebra $\End(\bigoplus_f P_f)$, given by
\[
E\mapsto \Hom(\bigoplus_f P_f,E).
\]
Moreover, the Yoneda $\Ext$ groups in this category agree with the $\Ext$ groups
in the category of sheaves.
\end{thm}

\begin{proof}
  The main claim follows from the fact that $\bigoplus_f P_f$ is a
  progenerator of the category (it is coherent, projective, and generates
  the category).  

  By using a projective resolution to compute the Yoneda $\Ext$ groups, we
  see that the remaining claim reduces to showing that the group
  $\Ext^p_{\Coh(X)}(P,E)$ vanishes for $p>0$, $P$ projective, and $E$
  invisible.  But this is true for any projective object of the form
  $P_f$, and thus for direct summands of direct sums of such sheaves.
\end{proof}

Since the canonical map $\pi^*\sO_Y\to \pi^!\sO_Y$ is injective (it has
invisible kernel, but $\sO_X$ is torsion-free), it gives rise to a
canonical divisor $e_\pi$ representing $K_X-\pi^*K_Y$, with $e_\pi$
supported on the exceptional locus.

\begin{cor}
  Every invisible sheaf is scheme-theoretically supported on a subscheme
  of $e_\pi$.
\end{cor}

\begin{proof}
Every invisible sheaf is a quotient of a sum of sheaves $P_f$, so it
suffices to consider those.  Then $P_f$ is contained in the kernel $E$ of the
natural map
\[
\pi^*\pi_*\sO_X(-f^\vee)\to \pi^!\pi_*\sO_X(-f^\vee).
\]
Starting with a locally free resolution
\[
0\to V\to W\to \pi_*\sO_X(-f^\vee)\to 0,
\]
we can use the snake lemma to obtain an exact sequence
\[
0\to E\to \pi^!V/\pi^*V\to \pi^!W/\pi^*W.
\]
Since
\[
\pi^!V/\pi^*V
\cong 
(\pi^!\sO_Y/\pi^*\sO_Y)\otimes \pi^*V
\cong
\sO_{e_\pi}(e_\pi)\otimes
\pi^*V,
\]
the result follows.
\end{proof}

\begin{rem}
  The invisible sheaf $\pi^!\sO_Y/\pi^*\sO_Y\cong \sO_{e_\pi}(e_\pi)$ shows
  that this bound on the support is tight.
\end{rem}

The injective objects in the category of invisible sheaves can also be
identified, most easily using the following duality, a special case of a
standard duality for sheaves of homological dimension 1 and codimension 1
support.

\begin{lem}
If $E$ is an invisible sheaf, then so is $\sExt^1(E,\omega_X)$, and
$\sExt^1(\sExt^1(E,\omega_X),\omega_X)\cong E$.
\end{lem}

\begin{proof}
  Since $E$ has $1$-dimensional support and homological dimension 1,
\[
\sHom(E,\omega_X)=\sExt^2(E,\omega_X)=0,
\]
and thus $\dR\sHom(E,\omega_X)$ is concentrated in degree 1.  Since
$\omega_X=\pi^!\omega_Y$, we have
\[
\dR\pi_*\dR\sHom(E,\omega_X)
\cong
\dR\sHom(\dR\pi_*E,\omega_Y)
= 0,
\]
and thus $\dR\pi_*\sExt^1(E,\omega_X)=0$ as required.

That this operation is a duality follows by considering how it acts on a
locally free resolution $0\to V\to W\to E\to 0$.
\end{proof}

\begin{rem}
This implies that the $k$-algebra associated to $\pi$ above is isomorphic
to its opposite.
\end{rem}

In particular, we can define injective objects
\[
I_f:= \sExt^1(P_f,\omega_X)\cong \sExt^1(\pi^*\pi_*\sO_X(-f^\vee),\omega_X)
\]
and find that $I_f$ is the injective hull of $\sO_f(-1)$.  Note that the
sheaves $\sO_f(-1)$ are self-dual: considerations of support show that the
dual of $\sO_f(-1)$ has the form $\sO_f(-d)$, and $\sO_f(-1)$ is the only
invisible possibility.  Another self-dual invisible sheaf is the sheaf
$\pi^!\omega_Y/\pi^*\omega_Y$: the dual is most naturally expressed as
$\pi^!\sO_Y/\pi^*\sO_Y$, but twisting by the pullback of a line bundle has
no effect on an invisible sheaf.

\section{Invisible sheaves and minimal lifts}

In the case of direct images of line bundles on smooth curves, the purpose
of lifting is to make the lifted sheaf disjoint from the anticanonical
curve.  This is of course not possible for sheaves of positive rank, but we
can still hope to make the restriction to the anticanonical curve simpler.
Thus we would like to understand how the restriction of the minimal lift is
related to the original restriction.  This is not quite the correct
question, as it turns out: it turns out to be easier to twist by a line
bundle and consider
\[
M\otimes \sO_{C_\alpha}(C_\alpha).
\]
The standard resolution
\[
0\to \sO_X\to \sO_X(C_\alpha)\to \sO_{C_\alpha}(C_\alpha)\to 0
\]
shows that
\[
\Tor_p(M,\sO_{C_\alpha}(C_\alpha))
\cong
\sExt^{1-p}_Y(\sO_{C_\alpha},M).
\]
Thus more generally we want to understand how
$\sExt^*_Y(\pi^{*!}M,\pi^{*!}N)$ and $\sExt^*_Y(M,N)$ are related.

\begin{lem}\label{lem:lift_exts}
Let $\pi:X\to Y$ be a birational morphism of smooth surfaces, and suppose
$M$, $N$ are sheaves on $Y$ with homological dimension $\le 1$.  Then
there is an isomorphism
\[
\sHom_Y(\pi^{*!}M,\pi^{*!}N)\cong \sHom_Y(M,N)
\]
and an exact sequence
\[
0\to \sExt^1_Y(\pi^{*!}M,\pi^{*!}N)
 \to \sExt^1_Y(M,N)
 \to \sHom_Y(E_1,E_2)
 \to \sExt^2_Y(\pi^{*!}M,\pi^{*!}N)
 \to 0,
\]
where $E_1$ and $E_2$ are the $\pi$-invisible sheaves fitting into the
short exact sequences
\begin{align}
&0\to E_1\to \pi^*M\to \pi^{*!}M\to 0\\
&0\to \pi^{*!}N\to \pi^!N\to E_2\to 0.
\end{align}
\end{lem}

\begin{proof}
Using the $E_1$ exact sequence, we obtain the long exact sequence
\[
\cdots\to
R^p\sHom_Y(\pi^{*!}M,\pi^{*!}N)
\to
R^p\sHom_Y(\pi^*M,\pi^{*!}N)
\to
R^p\sHom_Y(E_1,\pi^{*!}N)
\to
\cdots
\]
Moreover, $R^p\sHom_Y(\pi^*M,\pi^{*!}N)\cong R^p\sHom_Y(M,N)$ by
adjointness.

Similarly, since $R^p\sHom_Y(E_1,\pi^!N)=0$, the $E_2$ exact sequence gives
isomorphisms
\[
R^p\sHom_Y(E_1,E_2)
\cong
R^{p+1}\sHom_Y(E_1,\pi^{*!}N).
\]

Combining these gives a long exact sequence
\[
\cdots\to
R^p\sHom_Y(\pi^{*!}M,\pi^{*!}N)
\to
R^p\sHom_Y(M,N)
\to
R^{p-1}\sHom_Y(E_1,E_2)
\to
\cdots,
\]
from which the claim follows immediately.  Note that $\sExt^2_Y(M,N)=0$
since $M$ has homological dimension $\le 1$ by assumption.
\end{proof}

\begin{rem}
Similarly, we have an isomorphism
\[
\Hom(\pi^{*!}M,\pi^{*!}N)\cong \Hom(M,N)
\]
and a long exact sequence
\[
\cdots
\to
\Ext^{p-2}(E_1,E_2)
\to
\Ext^p(\pi^{*!}M,\pi^{*!}N)
\to
\Ext^p(M,N)
\to
\Ext^{p-1}(E_1,E_2)
\to
\cdots
\]
\end{rem}

\begin{cor}
Let $\pi:X\to Y$ be a birational morphism of smooth surfaces, and suppose
$M$ is a sheaf on $Y$ with homological dimension $\le 1$.
Then $\End(\pi^{*!}M)\cong \End(M)$.  In particular, the minimal lift of a
simple sheaf is always simple.
\end{cor}

\begin{cor}
  Let $\pi:X\to Y$ be a Poisson birational morphism of Poisson surfaces,
  and suppose $M$ is a sheaf on $Y$ of homological dimension $\le 1$.  Then
\[
\sHom_Y(\sO_{C_{\alpha_X}},\pi^{*!}M)\cong \sHom_Y(\sO_{C_{\alpha_Y}},M)
\]
and there is an exact sequence
\begin{align}
0\to \sExt^1_Y(\sO_{C_{\alpha_X}},\pi^{*!}M)
 &\to \sExt^1_Y(\sO_{C_{\alpha_Y}},M)\notag\\
 &\to \sHom_Y(\omega_X/\pi^*\omega_Y,\pi^!M/\pi^{*!}M)
 \to \sExt^2_Y(\sO_{C_{\alpha_X}},\pi^{*!}M)
 \to 0.\notag
\end{align}
\end{cor}

For a monoidal transformation, the invisible sheaves associated to a
minimal lift are straightforward to compute.

\begin{prop}
  Let $\pi:X\to Y$ be a monoidal transformation, and let $M$ be a sheaf on
  $Y$ of homological dimension $\le 1$.  Then there is an exact sequence
\[
0\to E_1\to \pi^*M\to \pi^!M\to E_2\to 0,
\]
where we have non-canonical isomorphisms
\begin{align}
E_1\cong \Ext^1(\sO_p,M)\otimes_k \sO_e(-1),\notag\\
E_2\cong \Hom_k(\Hom(M,\sO_p),\sO_e(-1)).\notag
\end{align}
\end{prop}

\begin{proof}
  There certainly is an exact sequence of the given form, since the kernel
  and cokernel are invisible, thus powers of $\sO_e(-1)$.  From the exact
sequence
\[
0\to E_1\to \pi^*M\to \pi^{*!}M\to 0,
\]
we find
\begin{align}
\Hom(\sO_e(-1),E_1)
\cong
\Hom(\sO_e(-1),\pi^*M)
&\cong
\Hom(\sO_e(-2),\pi^!M)\notag\\
&\cong
\Ext^1(R^1\pi_*\sO_e(-2),M)
\cong
\Ext^1(\sO_p,M).\notag
\end{align}
Similarly,
\[
\Hom(E_2,\sO_e(-1))
\cong
\Hom(\pi^!M,\sO_e(-1))
\cong
\Hom(\pi^*M,\sO_e)
\cong
\Hom(M,\sO_p).
\]
\end{proof}


This has a nice consequence for twists of minimal lifts, which we will
generalize below.

\begin{cor}
  Let $\pi:X\to Y$ be a monoidal transformation, and $M$ a sheaf on $Y$ of
  homological dimension $\le 1$.  Then the sheaves $\pi^{*!}M(\pm e)$ are
  $\pi_*$-acyclic, with direct image of homological dimension $\le 1$.
  Moreover, we have short exact sequences
\[
0\to M\to \pi_*(\pi^{*!}M(e))\to \Ext^1(\sO_p,M)\otimes_k \sO_p \to 0,
\]
and
\[
0\to \pi_*(\pi^{*!}M(-e))\to M\to 
\Hom_k(\Hom(M,\sO_p),\sO_p)\to 0.
\]
In other words, $\pi_*(\pi^{*!}M(e))$ is the universal
extension of $\sO_p$ by $M$, and $\pi_*(\pi^{*!}M(-e))$ is
the kernel of the universal homomorphism from $M$ to $\sO_p$.
\end{cor}

\begin{proof}
Take the short exact sequences
\begin{align}
0\to E_1\to \pi^*M\to \pi^{*!}M\to 0\notag\\
0\to \pi^{*!}M\to \pi^!M\to E_2\to 0\notag,
\end{align}
twist the first by $\sO_X(e)$ and the second by $\sO_X(-e)$, then
take (higher) direct images.  Note that after twisting, the middle terms in
the short exact sequence are still $\pi_*$-acyclic (since they are $\pi^!M$
and $\pi^*M$ respectively).

It remains only to show that $\pi^{*!}M(-e)$ is
$\pi_*$-acyclic and $\pi^{*!}M(e)$ has direct image of
homological dimension 1; the other two claims follow from the facts that
they are a subsheaf or quotient of a well-behaved sheaf.  But
\[
\Hom(\pi^{*!}M(-e),\sO_e(-1))
\cong
\Hom(\pi^{*!}M,\sO_e(-2))
\subset
\Hom(\pi^{*!}M,\sO_e(-1))
=0,
\]
and similarly for
\[
\Hom(\sO_e(-1),\pi^{*!}M(e)).
\]
\end{proof}

Note that if we twist by more than just $\pm e$, then we cannot expect to
obtain a well-behaved direct image.  Indeed, twisting $\sO_X\cong
\pi^{*!}\sO_Y$ by $\sO_X(2e)$ gives (naturally enough) $\sO_X(2e)$,
which fits in an exact sequence
\[
0\to \sO_X(e)\to \sO_X(2e)\to \sO_e(-2)\to 0
\]
so that $R^1\pi_*\sO_e(-2)\cong \sO_p$.  So we cannot iterate the twisting
itself, but must instead iterate the direct-image-of-twist-of-minimal-lift
operation (which we call a {\em pseudo-twist}).  Again, we must be careful
to note that the two pseudo-twists are not inverses to each other.  Indeed,
if $M$ is not torsion, then the pseudo-twist by $\sO_X(e)$ and the
pseudo-twist by $\sO_X(-e)$ change $[M]\in K_0$ by different multiples of
$[\sO_p]$, since
\begin{align}
\dim\Hom(M,\sO_p)
-
\dim\Ext^1(\sO_p,M)
&{}=
\dim\Ext^2(\sO_p,M)
-
\dim\Ext^1(\sO_p,M)\notag\\
&{}=
\dim\Ext^2(\sO_p,M)-\dim\Ext^1(\sO_p,M)+\dim\Hom(\sO_p,M)\notag\\
&{}=
\rank(M).\notag
\end{align}

Even if $M$ is torsion, they need not be inverses, e.g., if $M$ is an
invertible sheaf on a curve singular at $p$.  However, here the situation
is much nicer.  First, if $\pi^{*!}M$ is transverse to $e$, then no matter
how far we twist $\pi^{*!}M$ in either direction, the result will still be
a minimal lift, and thus in the transverse case, the pseudo-twists act as a
group.  More generally, the subsheaf $\Tor_1(\pi^{*!}M,\sO_e(-1))\subset
\pi^{*!}M$ will be a vector bundle on $e\cong \P^1$ having no global
sections.  If we repeatedly pseudo-twist by $\sO_X(-e)$, then each step
either strictly decreases the rank of this vector bundle or strictly
increases its degree.  Thus after a finite number of downwards
pseudo-twists, we obtain a sheaf such that $\pi^{*!}M$ is transverse to the
exceptional curve.  We could also obtain a sheaf with transverse minimal
lift via finitely many upwards pseudo-twists, or by some combination of the
two; the resulting sheaves need not be in the same orbit under the group of
pseudo-twists.

In generalizing pseudo-twists beyond monoidal transformations, we encounter
the problem of determining which twists of a minimal lift are guaranteed to
have well-behaved direct image (i.e., $\pi_*$-acyclic with direct image of
homological dimension $\le 1$).  This appears to be a tricky problem in
general (especially in the presence of exceptional components of
self-intersection $<-2$), but there is one sufficiently large special
case.

\begin{thm}
Let $\pi:X\to Y$ be a birational morphism of algebraic surfaces, and
suppose $D$ is a divisor on $X$ which is supported on the exceptional locus
and satisfies $D^2=-1$.  If $M$ is any sheaf of homological dimension
$\le 1$ on $Y$, then $\pi^{*!}M(D)$ is $\pi_*$-acyclic,
with direct image of homological dimension $\le 1$.
\end{thm}

\begin{proof}
Since the intersection pairing on the exceptional locus is integral and
negative definite, there are at most $2n$ exceptional divisors of
self-intersection $-1$, where $n$ is the number of monoidal transformations
in a factorization of $\pi$.  We can moreover construct those divisors
easily enough.  Given a factorization of $\pi$, let $e_i$ be the total
transform of the $i$-th exceptional curve through the remaining $n-i$
blowups; then $e_i^2=-1$, so $\pm e_i$ satisfy the hypotheses.

In other words, we need to show that $\pi^{*!}M(\pm e_i)$ have well-behaved
direct images.  Moreover, we may assume $i=n$, since twisting by $\sO_X(\pm
e_i)$ commutes with blowing down the last $n-i$ points.  Factor
$\pi=\pi_n\circ \pi_{[n]}$, where $\pi_n$ is the monoidal transformation
blowing down the $-1$-curve $e_n$.  Then $\pi^{*!}M(e_n)$ is a quotient of
$\pi_n^!\pi_{[n]}^*M$, so is $\pi_*$-acyclic, while $\pi^{*!}M(-e_n)$ is a
subsheaf of $\pi_n^*\pi_{[n]}^!M$, so has direct image of homological
dimension $\le 1$.

It remains to show that $\pi^{*!}M(-e_n)$ is $\pi_*$-acyclic and
$\pi^{*!}M(e_n)$ has direct image of homological dimension 1.  The key idea
is that not only is $e_n$ effective, but up to twisting by pullbacks of
line bundles, so is $-e_n$.  More precisely, on $\pi_n(X)$, we may choose a
divisor $D$ which is transverse to $\pi^{*!}_{[n]}M$ and meets the
exceptional locus only in $\pi_n(e_n)$ (this certainly exists in a formal
neighborhood of the exceptional locus).  Then the strict transform
$\pi^*D-e_n$ of $D$ is transverse to $\pi^{*!}M$, so that
\[
\pi^{*!}M(e_n)\subset\pi^{*!}(M(D)),
\]
implying that $\pi_*(\pi^{*!}M(e_n))$ has homological
dimension $\le 1$.  Similarly, $\pi^{*!}M(-e_n)$ is
a quotient of $\pi^{*!}(M(-D))$, so is $\pi_*$-acyclic.
\end{proof}

Note that changing the factorization of $\pi$ only permutes the divisors
$e_i$, since they are characterized as the unique divisors of
self-intersection $-1$ which are nonnegative linear combinations of
exceptional components.  Each divisor $e_i$ is the total transform of a
$-1$-curve, the strict transform of which is an exceptional component
$f_i$; this correspondence is canonical, so that we can more naturally
label the orthonormal divisors by exceptional components.  I.e., for any
exceptional component $f$, $e_f$ is the unique effective exceptional
divisor such that $e_f^2=-1$ and $e_f\cdot f=-1$.

\begin{cor}
Let $\pi:X\to Y$ be a birational morphism of algebraic surfaces, and let
$f$ be an exceptional component of $\pi$.  If $M$ is any sheaf of
homological dimension $\le 1$ on $Y$, then there are short exact sequences
of the form
\[
0\to M\to \pi_*(\pi^{*!}M(e_f))\to \sO_{\pi(f)}^{r_1}\to 0
\]
and
\[
0\to \pi_*(\pi^{*!}M(-e_f))\to M\to 
\sO_{\pi(f)}^{r_2}\to 0.
\]
The dimensions $r_1$, $r_2$ are given by
\[
r_1 = c_1(\pi^{*!}M)\cdot e_f,\qquad
r_2 = c_1(\pi^{*!}M)\cdot e_f+\rank(M).
\]
\end{cor}

\begin{proof}
The first short exact sequence comes from the direct image of the sequence
\[
0\to 
E_1(e_f)
\to
\pi^*M(e_f)
\to
\pi^{*!}M(e_f)
\to
0,
\]
and thus $r_1 = -\chi(E_1(e_f))$.  Since $\chi(E_1)=0$ and
$\rank(E_1)=0$, Hirzebruch-Riemann-Roch gives $r_1 = -c_1(E_1)\cdot e_f$.
Since
\[
c_1(E_1) = \pi^*c_1(M)-c_1(\pi^{*!}M),
\]
the formula for $r_1$ follows.

Similarly, the second short exact sequence is the direct image of
\[
0\to \pi^{*!}M(-e_f)\to \pi^!M(-e_f)
 \to E_2(-e_f)\to 0,
\]
so that $r_2 = \chi(E_2(-e_f))=-c_1(E_2)\cdot e_f$.
We find
\[
c_1(E_2) = c_1(\pi^!M)-c_1(\pi^{*!}M)
\]
and
\[
c_1(\pi^!M) = c_1(\pi^*M)+\rank(M)e_\pi.
\]
Since $e_\pi=\sum_f e_f$, we find $e_\pi\cdot e_f=-1$, and the formula
follows.
\end{proof}

Note that just as in the monoidal case, the upwards and downwards
pseudo-twists do not form a group unless $M$ is $1$-dimensional and
$\pi^{*!}M$ is transverse to the exceptional locus.  In addition, if $M$ is
$1$-dimensional and $\pi^{*!}M$ is not transverse to $e_\pi$, then we need
only finitely many pseudo-twists to make it transverse.  This has a
particularly nice consequence in the Poisson case.

\begin{lem}
  Let $(Y,\alpha)$ be a Poisson surface with anticanonical curve $C_\alpha$.
  Let $M$ be a pure $1$-dimensional sheaf on $Y$ (i.e., $M$ has
  $1$-dimensional support, and no subsheaf has $0$-dimensional support)
  which is transverse to $C_\alpha$.  Then there exists a Poisson birational
  morphism $\pi:X\to Y$ such that some pseudo-twist $M'$ of $M$ has minimal
  lift disjoint from the anticanonical curve on $X$.
\end{lem}

\begin{proof}
  If $\supp(M)$ is disjoint from $C_\alpha$, then there is nothing to prove.
  Otherwise, let $\pi_1:Y_1\to Y$ be the (Poisson) blow up in a point of
  intersection, and let $M_1$ be a pseudo-twist of $M$ such that
  $\pi^{*!}_1M_1$ is transverse to the exceptional locus.  Then
  $\pi^{*!}_1M_1$ is also transverse to $C_{\alpha_1}$, where $\alpha_1$ is the
  induced Poisson structure on $Y_1$.  Moreover,
\[
c_1(\pi^{*!}_1M_1)\cdot c_1(C_{\alpha_1})
=
(\pi^*c_1(M_1)-r_1e_1)(\pi^*c_1(C_\alpha)-e_1)
<
c_1(M_1)\cdot c_1(C_\alpha)
=
c_1(M)\cdot c_1(C_\alpha).
\]
(That $c_1(M_1)=c_1(M)$ is immediate from the short exact sequences for the
atomic pseudo-twists.)  Thus if we iterate this operation, then after at
most $c_1(M)\cdot c_1(C_\alpha)$ blowups, we will achieve disjointness.
\end{proof}

\begin{rem}
Since one of our motivations is to avoid using support (since this
makes little sense in a noncommutative context), we should note
that ``pure $1$-dimensional'' can be rephrased strictly in terms
of Hilbert polynomials: $M$ has linear Hilbert polynomial, and no
nonzero subsheaf has constant Hilbert polynomial.  Similarly, the
first Chern class of a pure $1$-dimensional sheaf specifies how the
linear term in the Hilbert polynomial depends on the choice of very
ample line bundle.
\end{rem}

If $\pi^{*!}M$ is disjoint from $C_{\alpha_X}$, then we find
\[
M\otimes \sO_{C_{\alpha_Y}}(C_{\alpha_Y})
\cong
\sExt^1_Y(\sO_{C_{\alpha_Y}},M)
\cong
\sHom_Y(\omega_X/\pi^*\omega_Y,\pi^!M/\pi^{*!}M).
\]
Moreover, disjointness has strong consequences for the invisible quotient
$\pi^!M/\pi^{*!}M$.  Since for any exceptional component, $f\cdot
C_{\alpha_X} = f^2+2$, any component with $f^2<-2$ is contained in
$C_{\alpha_X}$, and thus must be disjoint from $\pi^{*!}M$.

\begin{prop}
  Let $\pi:X\to Y$ be a birational morphism of algebraic surfaces.  If $M$
  is a sheaf of homological dimension $\le 1$ on $Y$ such that
  $\supp(\pi^{*!}M)$ does not contain any exceptional components of
  self-intersection $\le -3$, then $\pi^!M/\pi^{*!}M$ is an injective
  object of the category of $\pi$-invisible sheaves.  More precisely,
\[
\pi^!M/\pi^{*!}M
\cong
\bigoplus_f I_f^{-f\cdot c_1(\pi^{*!}M)-\rank(M)f\cdot e_\pi}.
\]
\end{prop}

\begin{proof}
Let $f$ be any exceptional component.  Since
\[
\dR\Hom(\sO_f(-1),\pi^!M)=0,
\]
we have
\begin{align}
\Ext^1(\sO_f(-1),\pi^!M/\pi^{*!}M)
&{}\cong
\Ext^2(\sO_f(-1),\pi^{*!}M)
\cong
H^1(\sExt^1_X(\sO_f(-1),\pi^{*!}M))\notag\\
&{}\cong
H^1(\sO_f(f^2+1)\otimes \pi^{*!}M)
\cong
H^0(R^1\pi_*(\sO_f(f^2+1)\otimes \pi^{*!}M))
\end{align}
If $\pi^{*!}M$ is transverse to $f$, the tensor product is $0$-dimensional,
so acyclic.  Otherwise, if $\sO_f(f^2+1)\otimes \pi^{*!}M$ (a sheaf
supported on $f\cong \P^1$) is not acyclic, then there is a nontrivial
morphism
\[
\sO_f(f^2+1)\otimes \pi^{*!}M\to \sO_f(-2),
\]
so a nontrivial morphism
\[
\pi^{*!}M\to \sO_f\otimes \pi^{*!}M\to \sO_f(-3-f^2),
\]
which is impossible unless $f^2\le -3$.

For the final claim, since $\pi^!M/\pi^{*!}M$ is injective, it is a direct
sum of injective hulls of simple sheaves, and $I_f$ occurs as a summand
with multiplicity $-f\cdot c_1(\pi^!M/\pi^{*!}M)$, since
$c_1(I_f)=c_1(P_f)=f^\vee$.  The claim then follows from
\[
c_1(\pi^!M/\pi^{*!}M)
=
\pi^*c_1(M) + \rank(M)e_\pi-c_1(\pi^{*!}M).
\]
\end{proof}

\begin{rem}
Note that the hypotheses are much weaker than disjointness; in particular,
if there are no exceptional components of self-intersection $\le -3$, then
$\pi^!M/\pi^{*!}M$ is injective as an invisible sheaf for {\em all}
sheaves $M$ of homological dimension $\le 1$.  Conversely, if $f^2\le -3$
for some $f$, then $\pi^!\sO_Y/\pi^*\sO_Y$ is not injective, since
\[
\Ext^1(\sO_f(-1),\pi^!\sO_Y/\pi^*\sO_Y)
\cong
\Ext^2(\sO_f(-1),\sO_X)
\cong
H^0(\sO_f(-1)\otimes \omega_X)^*
\cong
H^0(\sO_f(-3-f^2))^*
\ne 0.
\]
\end{rem}

In the disjoint case, we have
\[
M\otimes \sO_{C_{\alpha_Y}}
\cong
M\otimes \sO_{C_{\alpha_Y}}(C_{\alpha_Y})
\cong
\bigoplus_f \sHom_Y(\omega_X/\pi^*\omega_Y,I_f)^{-f\cdot
  c_1(\pi^{*!}M)}.
\]
Thus to complete our understanding of the disjoint case, we need to
understand the $0$-dimensional sheaves
\[
\sHom_Y(\omega_X/\pi^*\omega_Y,I_f)
\cong
\sHom_Y(P_f,\pi^!\sO_Y/\sO_X).
\]

\begin{prop}
  Let $\pi:X\to Y$ be a birational morphism of algebraic surfaces, let $f$
  be an exceptional component, and set $d=f^\vee\cdot e_\pi$.  Then there
  exists a homomorphism $\sO_{Y,\pi(f)}\to k[t]/t^d$ such that
\[
\sHom_Y(P_f,\pi^!\sO_Y/\sO_X)
\]
is the $\sO_Y$-module induced by the regular representation of $k[t]/t^d$.
\end{prop}

\begin{proof}
  Let $D$ be a smooth divisor on $X$ meeting the exceptional locus only in
  $f$, and that in a single reduced point $p$ (as before, we replace $X$ by
  a formal neighborhood of the exceptional locus if necessary).  Then
  $D=\pi^*\pi_*D-f^\vee$, so that $P_f$ fits into an exact sequence
\[
0\to P_f\to \pi^*\pi_*\sO_X(D)\to \sO_X(D)\to 0.
\]
Since $P_f$ is invisible, the canonical global section of $\sO_X(D)$
lifts to a canonical global section of $\pi^*\pi_*\sO_X(D)$, and we can
thus quotient by the two copies of $\sO_X$ to obtain
\[
0\to P_f\to \pi^*\pi_*(\sO_D(D))\to \sO_D(D)\to 0
\]
Since $\sO_D(D)$ is transverse to the exceptional locus, it
is a minimal lift, and thus we have an exact sequence
\begin{align}
0\to \sExt^1_Y(\sO_D(D),\sO_X)
 &{}\to \sExt^1_Y(\sO_D(D),\pi^!\sO_Y)\notag\\
 &{}\to \sHom_Y(P_f,\pi^!\sO_Y/\sO_X)
 \to \sExt^2_Y(\sO_D(D),\sO_X)
 \to 0\notag
\end{align}
Using the natural locally free resolution of $\sO_D(D)$
turns this into a short exact sequence
\[
0
\to 
\pi_*(\sO_D)
\to
\pi_*(\sO_D\otimes \pi^!\sO_Y)
\to
\sHom_Y(P_f,\pi^!\sO_Y/\sO_X)
\to
0;
\]
the assumptions on $D$ mean that $R^1\pi_*\sO_D=0$, and thus
$\sExt^2_Y(\sO_D(D),\sO_X)=0$.  In particular, the above
short exact sequence must agree with the direct image of the short exact
sequence
\[
0\to \sO_D\to \sO_D\otimes \pi^!\sO_Y\to J\to
0.
\]
Now, since $D$ is smooth and meets the exceptional locus in a single point,
$J$ is the structure sheaf of a jet, so has the form $k[t]/t^d$, where $t$
is a uniformizer of $D$ at $p$.  The claim then follows using the
composition
\[
\sO_{Y,p}\to \sO_{D,p}\to k[t]/t^d
\]
to compute the direct image.
\end{proof}

\begin{rem}
Note that the homomorphism to $k[t]/t^d$ need not be surjective, reflecting
the fact that $\sHom_Y(P_f,\pi^!\sO_Y/\sO_X)$ need not be the structure
sheaf of a jet.
\end{rem}

The sheaf $\pi^!\sO_Y/\sO_X$ also has some relevant structure.  If $\pi$
simply blows up a collection of distinct points on $Y$, then any choice of
ordering on those points gives a natural filtration of $\pi^!\sO_Y/\sO_X$.
Interestingly, we have the same freedom even for much more complicated
birational morphisms.

\begin{prop}
The invisible subsheaves of $\pi^!\sO_Y/\sO_X$ form a boolean lattice.
The maximal chains in the lattice are in natural correspondence with the
orderings on the exceptional components; if $f_1$, $f_2$,\dots,$f_n$
is such an ordering, then there is a unique maximal chain
\[
0=E_0\subset E_1\subset\cdots E_n=\pi^!\sO_Y/\sO_X
\]
of invisible sheaves such that
\[
c_1(E_i/E_{i-1})=e_{f_i}.
\]
Moreover, for any exceptional component $f$,
\[
\sHom_Y(P_f,E_i/E_{i-1})\cong \sO_{\pi(f)}^{f^\vee\cdot e_{f_i}}.
\]
\end{prop}

\begin{proof}
  The subsheaves of $\pi^!\sO_Y/\sO_X$ are in natural correspondence with
  the subsheaves of $\pi^!\sO_Y$ containing $\sO_X$.  If $M$ is such a
  subsheaf, then $M/\sO_X$ is invisible iff $\pi^!\sO_Y/M$ is
  invisible.  In particular, $\pi^!\sO_Y/M$ cannot have any
  $0$-dimensional subsheaves, so that it has homological dimension $\le 1$.
  Thus $M$ is locally free, and $M\cong \sO_X(D)$ for some exceptional
  divisor $D$.  Moreover, we compute
\[
0 = \chi(M/\sO_X) = D\cdot (e_\pi-D).
\]
In particular, if we express $D$ in terms of the orthonormal basis $e_f$,
then every coefficient must be either 0 or 1.  In other words,
\[
D = \sum_{f\in S} e_f
\]
for some subset $S$ of the set of exceptional components.  Since the sheaf
$\sO_D(D)$ has Euler characteristic 0 and is contained
in the invisible sheaf $\pi^!\sO_Y/\sO_X$, it too is invisible.  We
thus conclude that the invisible subsheaves of $\pi^!\sO_Y/\sO_X$ are in
order-preserving bijection with the sets of exceptional components.
Since $c_1(\sO_D(D))=D$, the claim about Chern classes
of subquotients in maximal chains follows.

If $E$ is a subquotient in some maximal chain, with $c_1(E)=g$, then
\[
E\cong \sO_{e_g}(\sum_{h\in S} e_h)
\]
where $S$ is any set of exceptional components containing $g$.  We thus
need to show that for any such sheaf, $\sHom_Y(P_f,E)$ is
scheme-theoretically supported on $\sO_{\pi(f)}$; that it has length
  $f^\vee\cdot e_g$ follows from a Chern class computation.  Factor
  $\pi=\pi_1\circ\pi_2$ with $\pi_1$ monoidal, having exceptional curve
  $e$.  If $f\ne e$, then
\[
P_f\cong \pi^*_1P_{\pi_1(f)},
\]
so that
\[
\sHom_Y(P_f,E)
\cong
\sHom_Y(P_{\pi_1(f)},\pi_{1*}E).
\]
If $g=e$, then $E\cong \sO_g(-1)$, and thus $\pi_{1*}E=0$.  Otherwise, $E$
has the form $\pi_1^*E'$ or $\pi_1^!E'$ (depending on whether $e\in S$),
where $E'$ is a sheaf of the same form in $\pi_1(X)$.  Either way, the
claim follows by induction.

If $f=e$, then we may choose a curve $D$ as above, so that
\[
\sHom_Y(P_f,E)\cong \sExt^1_Y(\sO_D(D),E)
              \cong \pi_*(\sO_D\otimes E).
\]
Since $E$ is an invertible sheaf on its support and $D$ is transverse to
that support, $\sHom_Y(P_f,E)$ depends only on the support.  Thus if $g\ne
e$, then we may assume $E\cong \pi_1^!E'$, and thus
\[
\sHom_Y(P_f,E)\cong \sHom_Y(\pi_{1*}P_f,E').
\]
Since $\pi_{1*}P_f$ is projective, the claim follows by induction.

Finally, if $f=e=g$, then $\sHom_Y(P_f,E)$ has length 1, so is necessarily
supported on a single point.
\end{proof}

We close the section with another result on resolving sheaves via minimal
lifts.  This does not involve pseudo-twists, but as a result requires
somewhat stronger hypotheses.  Here $\Fitt_0(M)$ denotes the $0$-th Fitting
scheme of $M$, which for $M$ pure $1$-dimensional is a canonical divisor
representing $c_1(M)$.

\begin{prop}
  Suppose $(Y,\alpha)$ is a Poisson surface, and $M$ is a pure
  $1$-dimensional sheaf on $Y$ such that the divisor $\Fitt_0(M)$ meets
  $C_\alpha$ in a disjoint union of jets.  Let $\pi:X\to Y$ be the minimal
  desingularization of the blowup of $Y$ in the intersection.  Then $\pi$
  is Poisson, and $\pi^{*!}M$ is disjoint from the induced anticanonical
  curve.
\end{prop}

\begin{proof}
  We can compute $\pi$ by repeatedly blowing up single points of
  intersection, so that $\pi$ is in particular Poisson.  The only issue is
  to show that the jet condition (and thus transversality) is preserved
  under the blowup $\pi:X\to Y$ in a point $p\in C_\alpha\cap\supp(M)$.
If $C_\alpha$ is smooth at $p$, the jet condition in a neighborhood of $p$ is
automatic, and this remains true after blowing up $p$.

Thus suppose $C_\alpha$ is singular at $p$.  Then we have
\[
M\otimes (\sO_X/{\cal I}_p^2)
\cong
M\otimes \sO_{C_\alpha}\otimes (\sO_X/{\cal I}_p^2),
\]
where ${\cal I}_p$ is the ideal sheaf of $p$.  Thus since $M\otimes
\sO_{C_\alpha}$ is a sum of jets, so is $M\otimes \sO_X/{\cal I}_p^2$.  But
then we can compute $\pi^{*!}M$ near $e$ using the minimal resolution of
$M$ over the local ring at $p$, and conclude that $\pi^{*!}M$ is transverse
to the exceptional curve (meeting it in the tangent vectors to the jets
through $p$), and thus is transverse to $\pi^{*!}\sO_{C_\alpha}$.

It follows that we have a short exact sequence
\[
0\to \pi_*\sExt^1_X(\pi^{*!}\sO_{C_\alpha},\pi^{*!}M)
 \to \sExt^1_Y(\sO_{C_\alpha},\pi^{*!}M)
 \to \sHom_Y(\sO_e(-1),\pi^!M/\pi^{*!}M)
 \to 0,
\]
and moreover that
\[
\sHom_Y(\sO_e(-1),\pi^!M/\pi^{*!}M)
\cong
\sHom_Y(\sExt^1_Y(\sO_{C_\alpha},\pi^{*!}M),\sO_p).
\]
In particular,
\[
\pi_*\sExt^1_X(\pi^{*!}\sO_{C_\alpha},\pi^{*!}M)
\]
is again a sum of jets, so that the same holds for
\[
\sExt^1_X(\pi^{*!}\sO_{C_\alpha},\pi^{*!}M).
\]
\end{proof}

\section{Maps between Poisson moduli spaces}

In this section, $\pi:X\to Y$ will be a Poisson birational morphism of
Poisson surfaces, which are now over a more general base scheme $S$.
For convenience of notation, we will silently identify sheaves on $S$ with
their pullbacks to $X$ or $Y$ as appropriate.

Although our goal is to show that $\pi^{*!}$ respects the Poisson
structure (in a suitable sense), we will in fact mainly focus on $\pi_*$
instead; since $\pi^{*!}$ is an inverse to $\pi_*$, it will be easy to
derive Poissonness of $\pi^{*!}$ from Poissonness of $\pi_*$, but the latter
will also let us deal with pseudo-twists.

The first issue is that direct images do not in general preserve flatness
of families.  This is, of course, just a relative version of semicontinuity
questions, but is particularly easy to deal with in our case.

\begin{lem}
  Let $M$ be a coherent sheaf on $X$, flat over $S$.  If $R^1\pi_*M$ is
  flat over $S$, then so is $\pi_*M$, and moreover the natural map
\[
\pi_*M\otimes N\to \pi_*(M\otimes N)
\]
is an isomorphism for all coherent sheaves $N$ on $S$.  In particular, this holds if
every fiber of $M$ is $\pi_*$-acyclic.
\end{lem}

\begin{proof}
For any coherent sheaf $N$ on $S$, we have an isomorphism
\[
\dR\pi_*M\otimes^{\dL} N\cong \dR\pi_*(M\otimes^{\dL} N)\cong \dR\pi_*(M\otimes N).
\]
Since $\pi$ has fibers of dimension $\le 1$, the spectral sequence for
$\dR\pi_*M\otimes^\dL N$ has entries in only two rows.  We thus obtain
isomorphisms
\[
\Tor_{p+2}(R^1\pi_*M,N)\cong \Tor_p(\pi_*M,N)
\]
for $p>0$, an isomorphism
\[
R^1\pi_*M\otimes N\cong R^1\pi_*(M\otimes N),
\]
and an exact sequence
\[
0\to \Tor_2(R^1\pi_*M,N)\to \pi_*M\otimes N\to \pi_*(M\otimes N)\to
\Tor_1(R^1\pi_*M,N)\to 0.
\]
The claim follows immediately.
\end{proof}

It follows that $\pi_*$ induces a morphism from an open subspace of $\Spl_X$
to $\Spl_Y$.  More precisely, we have the following.

\begin{lem}
  Let $\Spl_{X,\pi}\subset \Spl_X$ be the subspace of simple sheaves $M$
  which are $\pi_*$-acyclic and have simple direct image.  Then
  $\Spl_{X,\pi}$ is an open subspace, and $\pi_*$ induces a surjective
  morphism from $\Spl_{X,\pi}$ to $\Spl_Y$.
\end{lem}

\begin{proof}
  That $\pi_*$-acyclicity is an open condition follows from the fact that
  for any $S$-flat family $M$, the functor $R^1\pi_*$ commutes with taking
  fibers; thus the $\pi_*$-acyclic locus is just the complement of the
  image of the closed subscheme $\Fitt_0(R^1\pi_*M)$ under the proper map
  $X\to S$.  Then $\pi_*M$ is a flat family, and simplicity is an open
  condition on flat families \cite{AltmanAB/KleimanSL:1980}.

  For surjectivity, note that point sheaves in $\Spl_Y$ are direct images
  of point sheaves in $\Spl_X$, while any other simple sheaf is the direct
  image of its minimal lift.
\end{proof}

We will need one more fact in the proof of Poissonness, which we will also
use in the next section, so separate out from the proof.

\begin{lem}\label{lem:useful_factorization}
Let $M$ be a $\pi_*$-acyclic sheaf on $X$ with direct image of homological
dimension $\le 1$.  Then there is a commutative diagram
\[
\begin{CD}
\pi^!\pi_*M\otimes \pi^*\omega_Y @>1\otimes \pi^*\alpha_Y>> \pi^!\pi_*M\\
 @VVV @AAA\\
  M\otimes \omega_X @>1\otimes \alpha_X>> M,
\end{CD}
\]
where the first vertical map factors as
\[
\pi^!\pi_*M\otimes \pi^*\omega_Y
\to
\pi^*\pi_*M\otimes \omega_X
\to
M\otimes \omega_X.
\]
\end{lem}

\begin{proof}
Since $\pi$ is an isomorphism outside the exceptional locus, the diagram
can only fail to commute on the exceptional locus.  The failure to commute
is thus measured by a morphism
\[
\pi^!\pi_*M\otimes \pi^*\omega_Y \to \pi^!\pi_*M
\]
that vanishes outside the exceptional locus.  We can thus apply Proposition
\ref{prop:invisible_image} to conclude that this morphism has invisible
image.  Since $\pi^!\pi_*M$ has no invisible subsheaf, the image is 0,
and thus the diagram commutes.
\end{proof}

\begin{thm}\label{thm:direct_image_is_Poisson}
  Let $\pi:X\to Y$ be a Poisson birational morphism of Poisson surfaces.
  The direct image functor $\pi_*$ defines a Poisson morphism
  $\pi_*:\Spl_{X,\pi}\to \Spl_Y$.
\end{thm}

\begin{proof}
  We have already shown that it defines a morphism, so it remains only to
  show that it respects the Poisson structure.  Consider a sheaf $M$ in
  $\Spl_{X,\pi}$.  If $\pi_*M$ has homological dimension $2$, then $\pi_*M$
  must be a point, and either $M$ is a point or $M$ is supported on the
  exceptional locus.  On the subspace $X\subset \Spl_X$ parametrizing point
  sheaves, $\pi_*$ is just $\pi$, so is certainly Poisson.  If $M$ is
  supported on the exceptional locus, then $\pi_*M$ is supported on the
  anticanonical curve, so that for such sheaves, $\pi_*$ maps to a single
  point with trivial Poisson structure, so again is Poisson.

  We may thus restrict our attention to the case that $\pi_*M$ has
  homological dimension $\le 1$.  Now, the adjunction between $\pi^!$ and
  $\pi_*$ induces natural maps
\begin{align}
  \Ext^1(\pi_*M,\pi_*M\otimes \omega_Y)&\cong \Ext^1(M,\pi^!\pi_*M\otimes \pi^*\omega_Y)\notag\\
  \Ext^1(\pi_*M,\pi_*M)&\cong \Ext^1(M,\pi^!\pi_*M),\notag
\end{align}
with inverse given by $\pi_*$.  (To be precise, given a $\pi_*$-acyclic
sheaf $N$, there is a natural map $\pi_*:\Ext^1(M,N)\to
\Ext^1(\pi_*M,\pi_*N)$ given by taking the direct image of the extension.)
In particular, the Poisson structure on a neighborhood of $\pi_*M$ in
$\Spl_Y$ is obtained by composing the above isomorphisms with the map
\[
\Ext^1(M,\pi^!\pi_*M\otimes \pi^*\omega_Y)
\xrightarrow{\Ext^1(M,1\otimes \pi^*\alpha_Y)}
\Ext^1(M,\pi^!\pi_*M).
\]
By Lemma \ref{lem:useful_factorization}, this factors through the natural map
\[
\Ext^1(M,M\otimes \omega_X)
\xrightarrow{\Ext^1(M,1\otimes \alpha_X)}
\Ext^1(M,M),
\]
and the composition
\[
\Ext^1(M,M) \to \Ext^1(M,\pi^!\pi_*M)\cong \Ext^1(\pi_*M,\pi_*M)
\]
is just $\pi_*$, i.e., the differential of the morphism
$\pi_*:\Spl_{X,\pi}\to \Spl_Y$.  In other words, the map
$\Omega_{\Spl_Y}\to \Omega^*_{\Spl_Y}$ induced by the Poisson structure is
the direct image of the corresponding map on $\Spl_{X,\pi}$, and thus
$\pi_*$ is Poisson.
\end{proof}

\begin{cor}
The functors $\pi^*$ and $\pi^!$ induce Poisson morphisms from $\Spl_Y$ to
$\Spl_{X,\pi}$.
\end{cor}

\begin{proof}
Indeed, $\pi^*$ and $\pi^!$ are injective, and preserve simplicity and
flatness.  Since the images of $\pi^*$ and $\pi^!$ satisfy the hypotheses
of the Theorem, $\pi_*$ is Poisson on both images; since it is also an
isomorphism on both images, the inverses are Poisson.  But this is
precisely what we want to prove.
\end{proof}

\medskip

For the minimal lift, we again have an issue with flatness.  In this case,
however, we can completely control the corresponding flattening
stratification.  Let $\Spl_{Y,\le 1}$ denote the subspace parametrizing
sheaves of homological dimension $\le 1$.

\begin{lem}
  Suppose $\pi$ is monoidal, blowing up the point $p\in Y$.  Then for each
  integer $m\ge 0$, the subspace of $\Spl_{Y,\le 1}$ parametrizing sheaves
  $M$ with $\dim(\Ext^1(\sO_p,M))=m$ is a Poisson subspace, and contains
  every symplectic leaf it intersects.
\end{lem}

\begin{proof}
Since $\pi$ is Poisson, $p$ is contained in the anticanonical curve
$C_\alpha$, and thus
\[
M\otimes \sO_p\cong (M\otimes \sO_{C_\alpha})\otimes \sO_p
\]
is constant on symplectic leaves, so determines a stratification by locally
closed Poisson subspaces.  Since
\[
\dim(H^0(M\otimes \sO_p))
=
\dim(\Hom(M,\sO_p))
=
\dim(\Ext^1(\sO_p,M))+\rank(M),
\]
the same applies to $\dim(\Ext^1(\sO_p,M))$.
\end{proof}

Now, given any flat family of sheaves on $Y$ with homological dimension
$\le 1$, we obtain a stratification of the base of the family by taking the
flattening stratification of the minimal lift of the family.  Normally this
depends on a choice of relatively very ample bundle, but one can always
refine such a stratification by imposing the open and closed conditions
that the numerical Chern classes be constant, and the result will then be
independent of the relatively very ample bundle.  It is this canonical
stratification that we will mean; it refines the usual decomposition of $Y$
by numerical Chern class.
 
\begin{cor}
  If $\pi$ is monoidal, then the stratification induced on $\Spl_{Y,\le 1}$
  by $\pi$ agrees with the stratification by $\dim\Ext^1(\sO_p,M)$ and
  numerical Chern class.
\end{cor}

\begin{proof}
Indeed, for any coherent sheaf $M$ of homological dimension $\le 1$, we
have the equation
\[
[\pi^{*!}M]=\pi^*[M]-\dim\Ext^1(\sO_p,M)[\sO_e(-1)]
\]
in the Grothendieck group.  It follows that $\pi^{*!}M$ has constant
numerical Chern class iff $M$ has constant numerical Chern class and
$\dim\Ext^1(\sO_p,M)$ is constant.
\end{proof}

\begin{thm}
  Let $\pi:X\to Y$ be a Poisson birational morphism of Poisson surfaces.
  Every stratum of the stratification induced on $\Spl_{Y,\le 1}$ by $\pi$
  is a Poisson subspace, and the minimal lift induces a Poisson isomorphism
  from each stratum to an open subspace of $\Spl_{X,\le 1}$.
\end{thm}

\begin{proof}
  First suppose that $\pi$ is monoidal.  Then we have already shown that
  the stratification is Poisson.  On each stratum, the minimal lift
  preserves flatness (and simplicity), so defines a morphism from the
  stratum to $\Spl_{X,\le 1}$.  That the image is open follows from
  semicontinuity and the fact that minimal lifts are characterized by the
  vanishing of $\Hom(\sO_e(-1),M)$ and $\Hom(M,\sO_e(-1))$.

  We thus obtain an isomorphism between each stratum and the corresponding
  open subspace, and it remains only to show that it is Poisson.
  Equivalently, we need to show that $\pi_*$ is Poisson on the image of
  $\pi^{*!}$, but this is immediate from Theorem
  \ref{thm:direct_image_is_Poisson}.

  More generally, if we factor $\pi=\pi_1\circ\pi_2$ with $\pi_1$ monoidal,
  then the stratification induced by $\pi$ refines the stratification
  induced by $\pi_2$, and is identified via $\pi_2^{*!}$ with the
  stratification induced by $\pi_1$ on $\pi_1(X)$.  The claim then follows
  easily by induction.
\end{proof}

\begin{cor}
  Let $\pi:X\to Y$ be a birational morphism of Poisson surfaces.  Then for
  every stratum of the stratification induced by $\pi$ on $\Spl_{Y,\le 1}$,
  any pseudo-twist operation defines a Poisson morphism on the open
  subspace of the stratum where the pseudo-twist is simple.
\end{cor}

\begin{proof}
Indeed, a pseudo-twist is a composition of the Poisson morphism
$\pi^{*!}$, the Poisson morphism ${-}\otimes\sO_X(\pm e_f)$, and
the (partially defined) Poisson morphism $\pi_*$.
\end{proof}

\begin{cor}
  Let $(Y,\alpha)$ be a Poisson surface, let the subscheme $J\subset C_\alpha$
  be a disjoint union of jets, and let $\pi:X\to Y$ be the minimal
  desingularization of the blowup of $Y$ along $J$.  For any sheaf $M_\alpha$
  which is a direct sum of structure sheaves of subschemes of $J$,
  $\pi^{*!}$ is a symplectomorphism on the symplectic leaf
  $M\otimes \sO_{C_\alpha}=M_\alpha$ of $\Spl_{Y,\le 1}$.
\end{cor}

\begin{proof}
Indeed, $M\otimes \sO_{C_\alpha}=M_\alpha$ iff $\pi^{*!}M$ is disjoint from
$C_\alpha$, and has the correct numerical Chern class.
\end{proof}

\begin{rem}
In particular, if $C_\alpha$ is smooth, then any symplectic leaf consisting
of sheaves transverse to $C_\alpha$ is symplectomorphic to an open subspace of
some $\Spl_X$ with $\pi:X\to Y$ a Poisson birational morphism: simply choose
$\pi:X\to Y$ so that $\pi^{*!}M$ is disjoint from $C_\alpha$.
\end{rem}

\medskip

The fact that inverse images and minimal lifts behave nicely with respect
to the Poisson structure has an interesting consequence.  Let $f:X\ratto Y$
be a Poisson birational map.  Then we can use the above facts to
essentially embed $\Spl_X$ as a Poisson subspace of $\Vect_Y$.  More
precisely, suppose $U\to \Spl_{X\le 1}$ is a Noetherian \'etale
neighborhood (in the remaining case, when $U$ is an \'etale neighborhood of
$X$, simply consider the corresponding family of ideal sheaves).  Then
there is a corresponding morphism $U\to \Vect_Y$ (\'etale to its image)
such that the two induced Poisson structures are the same.

To see this, let $Z$ be a Poisson resolution of $f$, so that $f$ factors
through Poisson birational morphisms $g:Z\mapsto X$ and $h:Z\mapsto Y$.
Then if $M$ is the original flat family of sheaves on $X$, $g^* M$ is a
flat family of sheaves on $Z$, with the same Poisson structure.  We may
then apply the construction of \cite{HurtubiseJC/MarkmanE:2002b} (see
Section \ref{sec:moduli_hd1} above) to turn this into a flat family of
simple locally free sheaves on $Z$ (negating the Poisson structure in the
process).  Let $V$ be this family.  Now, if we twist $V$ by a suitable line
bundle, we may arrange to have
\[
\Hom(V,\sO_e(-1))=0
\]
for every $h$-exceptional component $e$. Indeed, we have
\[
\Hom(V,\sO_e(-1))
\cong
\Hom(V|_e,\sO_e(-1))
\cong
H^1(V|_e(-1))^*
\]
so we need simply twist by a relatively ample bundle that makes
$\bigoplus_e V|_e(-1)$ acyclic.  But this makes $V$ a minimal lift under
$h$!  Indeed, since $V$ is torsion-free, we also have $\Hom(\sO_e(-1),V)=0$
for all $e$, and thus Theorem \ref{thm:minimal_if_no_exceptional} applies.
In particular, it follows that $h_*V$ is simple, and induces an
anti-Poisson map from $U$ to $\Spl_{Y,\le 1}$.  Applying the construction
of \cite{HurtubiseJC/MarkmanE:2002b} again gives the required Poisson
morphism $U\to \Vect_Y$.  Each step of the above process is reversible, so
the maps $U\to \Spl_{X\le 1}$ and $U\to \Vect_Y$ have isomorphic images
(algebraic subspaces of $\Spl_{X\le 1}$, $\Vect_Y$ respectively).

One possible application of this construction is in the noncommutative
context.  The discussion of Section \ref{sec:moduli_hd1} essentially never
uses the fact that $X$ is a {\em commutative} surface; the main change
needed (assuming $X$ reasonably well-behaved) is to replace the functor
${-}\otimes \omega_X$ by the appropriate analogue for noncommutative Serre
duality.\footnote{There are also some minor technicalities involving the
  fact that twisting by a line bundle changes the noncommutative surface;
  of course, one can always twist back.}  The only other significant change
is that $V$ is no longer locally free in any obvious sense, but we can
still hope to show that some twist of $V$ has well-behaved direct image
(i.e., $V$ is $h_*$-acyclic, with simple image from which we can
reconstruct $V$), in which case the above construction can still be carried
out.\footnote{Of course, it would be enough to consider the case that $f$
  is monoidal or inverse-monoidal.}  The construction would then give us
the ability to finesse a difficulty of the noncommutative context, namely
the fact that it is unclear how to define ``locally free'' on a general
noncommutative surface, and thus even less clear how to prove that the
corresponding moduli space is Poisson.  This construction would mean it
sufficed to show the Poisson property for only one surface from each
birational equivalence class (e.g., for noncommutative ruled surfaces).  In
particular, the case of noncommutative rational surfaces would reduce to
that of noncommutative $\P^2$, in which case one may use more ad-hoc
calculations \cite{P2Painleve} to prove the Jacobi identity and classify
the symplectic leaves.

\bigskip

In addition to images and lifts, there is one more natural morphism we want
to consider.  We have already discussed the dualization functor
\[
{-}^D:=\sExt^1_Y({-},\omega_Y)
\]
in the context of invisible sheaves, but of course much of what we have
said applies to any pure $1$-dimensional sheaf.  Duality also interacts
nicely with the inverse image and minimal lift functors.

\begin{prop}
Let $\pi:X\to Y$ be a birational morphism of smooth projective surfaces.
Then for any pure $1$-dimensional sheaf $M$ on $Y$, $\pi^*M$, $\pi^!M$,
and $\pi^{*!}M$ are pure $1$-dimensional sheaves, and we have
\begin{align}
(\pi^*M)^D&\cong \pi^!(M^D)\notag\\
(\pi^{*!}M)^D&\cong \pi^{*!}(M^D)\notag\\
(\pi^!M)^D&\cong \pi^*(M^D).\notag
\end{align}
\end{prop}

\begin{proof}
This is essentially by inspection using a locally free resolution
\[
0\to V\to W\to M\to 0
\]
of $M$, and the corresponding resolution
\[
0\to W^*\to V^*\to M^D\to 0
\]
of $M^D$.
\end{proof}

Since duality interacts nicely with these Poisson morphisms, it is natural
to suspect that it is itself Poisson. This is not quite the case: it is in
fact anti-Poisson.

\begin{prop}
Let $(X,\alpha)$ be a Poisson surface, and let $\Spl^1_X$ be the subspace
parametrizing pure $1$-dimensional sheaves.  Then ${-}^D$
is an anti-Poisson involution on $\Spl^1_X$.
\end{prop}

\begin{proof}
Recall that at a point $M\in \Spl^1_X$, the pairing on the sheaf of
differentials is given by the composition
\[
\Ext^1(M,M\otimes \omega_X)\otimes \Ext^1(M,M\otimes \omega_X)
\to
\Ext^2(M,M\otimes \omega_X^2)
\to
\Ext^2(M,M\otimes \omega_X)
\to
k.
\]
Now, ${-}^D$ induces an isomorphism
\[
\Ext^1(M,M\otimes \omega_X)
\to
\Ext^1(M^D\otimes \omega_X^{-1},M^D)
\to
\Ext^1(M^D,M^D\otimes \omega_X).
\]
Indeed, a class in $\Ext^1(M,M\otimes \omega_X)$ corresponds to an
extension
\[
0\to M\otimes \omega_X\to N\to M\to 0.
\]
Since $N$ is an extension of pure $1$-dimensional sheaves, it is also pure
$1$-dimensional, and thus we can dualize the entire exact sequence to
obtain
\[
0\to M^D\to N^D\to (M\otimes \omega_X)^D\to 0,
\]
and note that $(M\otimes \omega_X)^D\cong M^D\otimes \omega_X^{-1}$.

Similarly, the product of two such extensions is given by a four-term exact
sequence of pure $1$-dimensional sheaves, which we can again dualize.  If
$\phi$, $\psi$ are two classes in $\Ext^1(M,M\otimes \omega_X)$, then we have
\[
(\phi\psi)^D = \psi^D\phi^D = -\phi^D\psi^D.
\]
Duality respects the trace map, so all told gives an anti-Poisson involution.
\end{proof}

\section{Rigidity}

In studying moduli spaces of differential and difference equations,
one particularly interesting question is when those moduli spaces
consist of a single point.  For instance, the hypergeometric
differential equation has been known since Riemann to be the unique
second-order Fuchsian equation with three singular points (and any
specific equation is determined by the exponents at the singularities).
In a generalized Hitchin system context, the order and singularity
structure precisely specifies a symplectic leaf in the corresponding
Poisson moduli space.  We are thus interested in understanding which
sheaves are rigid, i.e., the only point of their symplectic leaf.
More precisely, we say that a sheaf $M$ of homological dimension
$\le 1$ is rigid if it is isomorphic to any sheaf of homological
dimension $\le 1$ with the same numerical Chern class and (derived)
restriction to $C_\alpha$.  (This is stronger than the symplectic
leaf condition, since the symplectic leaf by definition includes
only simple sheaves.)

A slightly easier question is infinitesimal rigidity: at which sheaves is
the tangent space to the symplectic leaf $0$-dimensional?  Of course,
a sheaf is infinitesimally rigid iff the map
\[
\Ext^1(M,M\otimes \omega_X)\to \Ext^1(M,M)
\]
is $0$, or equivalently if the image is $0$-dimensional.  As usual, it will
be convenient to embed this in an Euler characteristic; we also extend to
pairs of sheaves.  Thus for two coherent sheaves $M$, $N$ on the Poisson surface
$(X,\alpha)$, we define
\[
\chi_\alpha(M,N)
:=
\dim(\Hom(M,N))
-
\dim(\im(\Ext^1(M,N\otimes \omega_X)\to \Ext^1(M,N)))
+
\dim(\Hom(N,M)).
\]
Note that duality immediately gives $\chi_\alpha(M,N)=\chi_\alpha(N,M)$;
it also gives
\[
\dim(\Hom(N,M))=\dim(\Ext^2(M,N\otimes \omega_X)),
\]
making $\chi_\alpha$ look more like an Euler characteristic.

\begin{prop}\label{prop:symp_dims}
  The $\alpha$-twisted Euler characteristic $\chi_\alpha(M,N)$ depends only
  on the numerical Chern classes $c_*(M)$, $c_*(N)$ and the derived
  restrictions $M|^{\dL}_{C_\alpha}$ and $N|^{\dL}_{C_\alpha}$.
\end{prop}

\begin{proof}
We first note that we have a long exact sequence
\[
0\to C^{-1}
 \to A^0
 \to B^0
 \to C^0
 \to A^1
 \to B^1
 \to C^1
 \to A^2
 \to B^2
 \to C^2
 \to 0
\]
where
\begin{align}
A^p &:= \Ext^p(M,N\otimes\omega_X),\notag\\
B^p &:= \Ext^p(M,N),\notag\\
C^p &:= \Ext^p(M|^{\dL}_{C_\alpha},N|^{\dL}_{C_\alpha}).\notag
\end{align}
Indeed, since the complex $N\otimes\omega_X\to N$ represents
$N\otimes^{\dL} \sO_{C_\alpha}$, to construct such a sequence, it will
suffice to give a quasi-isomorphism
\[
\dR\Hom(M,N\otimes^{\dL} \sO_{C_\alpha})
\cong
\dR\Hom(M|^{\dL}_{C_\alpha},N|^{\dL}_{C_\alpha}).
\]
If $i:C_\alpha\to X$ is the natural closed embedding, then
\[
N\otimes^{\dL} \sO_{C_\alpha} \cong \dR i_*\dL i^* N,
\]
and thus
\[
\dR\Hom(M,N\otimes^{\dL} \sO_{C_\alpha})
\cong
\dR\Hom(M,\dR i_* \dL i^* N)
\cong
\dR\Hom(\dL i^* M,\dL i^* N),
\]
Since $\dL i^*M\cong M|^{\dL}_{C_\alpha}$, the claim follows.

Then, by definition,
\[
\chi_\alpha(M,N) = \dim(B^0) -\dim(\im(A^1\to B^1)) + \dim(A^2),
\]
and exactness gives
\[
\dim(\im(A^1\to B^1))
=
\dim(A^1)-\dim(C^0)+\dim(B^0)-\dim(A^0)+\dim(C^{-1}),
\]
so that
\[
\chi_\alpha(M,N)
=
\dim(A^0)-\dim(A^1)+\dim(A^2)
+
\dim(C^0)-\dim(C^{-1}).
\]
The $C$ dimensions manifestly only depend on $M|^{\dL}_{C_\alpha}$ and
$N|^{\dL}_{C_\alpha}$, while
\[
\dim(A^0)-\dim(A^1)+\dim(A^2)
=
\dim\Ext^2(N,M)-\dim\Ext^1(N,M)+\dim\Hom(N,M)
\]
is the usual $\Ext$ Euler characteristic, and thus only depends on the
numerical Chern classes of $M$ and $N$ (and can be computed via
Hirzebruch-Riemann-Roch).
\end{proof}

\begin{rem}
For convenient reference, we note that
\begin{align}
\dim\Ext^2(N,M)&{}-\dim\Ext^1(N,M)+\dim\Hom(N,M)\notag\\
{}={}&
-\chi(\sO_X)\rank(M)\rank(N)
+\chi(M)\rank(N)
+\rank(M)\chi(N)\notag\\
&-c_1(M)\cdot c_1(N)
+\rank(M) K_X\cdot c_1(N).
\end{align}
If $\Tor_1(M,\sO_{C_\alpha})=\Tor_1(N,\sO_{C_\alpha})=0$, then $\chi_\alpha(M,N)$
is given by the above together with the contribution
$\dim\Hom(M|_{C_\alpha},N|_{C_\alpha})$ from the restriction to the
anticanonical curve.  Particularly nice is the case $\rank(M)=\rank(N)=0$,
in which case
\[
\chi_\alpha(M,N)
=
-c_1(M)\cdot c_1(N)
+
\dim\Ext^0(M|^{\dL}_{C_\alpha},N|^{\dL}_{C_\alpha})
-
\dim\Ext^{-1}(M|^{\dL}_{C_\alpha},N|^{\dL}_{C_\alpha}).
\]
\end{rem}

In the case $M=N$, we call $\chi_\alpha(M,M)$ the {\em index of rigidity} of
$M$, following \cite{KatzNM:1996}; note
\[
\chi_\alpha(M,M)
=
2\dim(\End(M))-\dim(\im(\Ext^1(M,M\otimes \omega_X)\to \Ext^1(M,M))).
\]
In particular, if $M$ is simple, then it is infinitesimally rigid iff
$\chi_\alpha(M,M)=2$.

The $\alpha$-twisted Euler characteristic, and thus the index of rigidity,
behaves nicely under direct images and minimal lifts.  Actually, the middle
term (which is the one of greatest interest in any case) has the weakest
hypotheses to ensure good behavior.

\begin{prop}
  Let $\pi:X\to Y$ be a Poisson birational morphism of Poisson surfaces,
  and suppose $M$ and $N$ are coherent sheaves on $X$, at least one of
  which is $\pi_*$-acyclic with direct image of homological dimension $\le
  1$.  Then
\[
\dim(\im(\Ext^1(\pi_*M,\pi_*N\otimes \omega_Y)\to \Ext^1(\pi_*M,\pi_*N)))
\le 
\dim(\im(\Ext^1(M,N\otimes \omega_X)\to \Ext^1(M,N))).
\]
\end{prop}

\begin{proof}
  By symmetry, we may assume that $N$ is $\pi_*$-acyclic with direct image
  of homological dimension $\le 1$.  Then, using adjointness, we observe
  that the maps
\[
\Ext^1(\pi_*M,\pi_*N\otimes \omega_Y)
\to
\Ext^1(\pi_*M,\pi_*N)
\]
and
\[
\Ext^1(M,\pi^!\pi_*N\otimes \pi^*\omega_Y)
\to
\Ext^1(M,\pi^!\pi_*N)
\]
have isomorphic images.  By Lemma \ref{lem:useful_factorization}, the
latter factors through the map
\[
\Ext^1(M,N\otimes \omega_X)
\to
\Ext^1(M,N),
\]
giving us a bound on the image.
\end{proof}

\begin{prop}
  Let $\pi:X\to Y$ be a birational morphism of algebraic surfaces, and
  suppose $M$ and $N$ are coherent sheaves on $X$ such that $M$ is
  $\pi_*$-acyclic, $\pi_*N$ has homological dimension 1, and either $M$ has
  no invisible quotient or $N$ has no invisible subsheaf.  Then
  $\dim\Hom(\pi_*M,\pi_*N)\ge \dim(\Hom(M,N))$.
\end{prop}

\begin{proof}
It suffices to show that the natural map $\Hom(M,N)\to \Hom(\pi_*M,\pi_*N)$ 
is injective.  Suppose $f:M\to N$ is in the kernel.  Then it is 0 outside
the invisible locus, so has invisible image, giving both an invisible
quotient of $M$ and an invisible subsheaf of $N$.
\end{proof}

Combining hypotheses gives the following.

\begin{cor}
  Let $M$, $N$ be coherent $\pi_*$-acyclic sheaves on $X$ with direct
  images of homological dimension $\le 1$.  If (1) $M$ has no
  invisible quotient or $N$ has no invisible subsheaf, and (2) $N$ has
  no invisible quotient or $M$ has no invisible subsheaf, then
  $\chi_\alpha(M,N)\le \chi_\alpha(\pi_*M,\pi_*N)$.
\end{cor}

\begin{cor}
  Let $M$ be a coherent $\pi_*$-acyclic sheaf on $X$ with direct image of
  homological dimension $\le 1$.  Suppose further that either $M$ is
  simple, $M$ has no invisible quotient, or $M$ has no invisible
  subsheaf.  Then $\chi_\alpha(M,M)\le \chi_\alpha(\pi_*M,\pi_*M)$.
\end{cor}

\begin{proof}
If $M$ has no invisible quotient or no invisible subsheaf, this is just
a special case of the proposition.  Those hypotheses were only used to
prove the inequalities on $\Hom$ spaces, which are automatic when $M$ is
simple.
\end{proof}

\begin{prop}
  If $M$, $N$ are coherent sheaves on $Y$ of homological dimension $\le 1$,
  then the minimal lift respects the $\alpha$-twisted Euler characteristic:
  $\chi_\alpha(\pi^{*!}M,\pi^{*!}N)=\chi_\alpha(M,N)$.
\end{prop}

\begin{proof}
From the remark following Lemma \ref{lem:lift_exts}, we know that
\[
\Hom(\pi^{*!}M,\pi^{*!}N)\cong \Hom(M,N)
\]
and
\[
\Ext^1(\pi^{*!}M,\pi^{*!}N)\subset \Ext^1(M,N).
\]
Dually, the map
\[
\Ext^1(M,N\otimes \omega_Y)
\to
\Ext^1(\pi^{*!}M,\pi^{*!}N\otimes \omega_X)
\]
is surjective.  But then the maps
\[
\Ext^1(\pi^{*!}M,\pi^{*!}N\otimes \omega_X)
\to
\Ext^1(\pi^{*!}M,\pi^{*!}N)
\]
and
\[
\Ext^1(M,N\otimes \omega_Y)
\to
\Ext^1(M,N)
\]
have isomorphic images.
\end{proof}

\begin{prop}
Let $\pi:X\to Y$ be a birational morphism of Poisson surfaces, and $M$ a
coherent sheaf on $Y$ of homological dimension $\le 1$.  If $M'$ is any
pseudo-twist of $M$, then $\chi_\alpha(M,M)\le \chi_\alpha(M',M')$.
\end{prop}

\begin{proof}
  Since
  $\chi_\alpha(M,M)=\chi(\pi^{*!}M,\pi^{*!}M)=\chi_\alpha(\pi^{*!}M(\pm
  e_f),\pi^{*!}M(\pm e_f))$, it remains to show that the last index of
  rigidity is nonincreasing under direct images.  That this holds for the
  $\Ext^1$ term follows from the fact that $\pi^{*!}M(\pm e_f)$ is
  $\pi_*$-acyclic with direct image of homological dimension $\le 1$.  For
  the $2\dim(\End(M))$ term, consider the composition
\[
\End(M)\cong \End(\pi^{*!}M)\cong \End(\pi^{*!}M(\pm
e_f))\to \End(M').
\]
This is an isomorphism away from the point $\pi(f)$, and thus
any endomorphism in the kernel must vanish away from this point.  In
particular, any such endomorphism would have $0$-dimensional image, which
must be 0 since $M$ has homological dimension $\le 1$.
\end{proof}

\medskip

In the Hitchin system context, the main question is which pure
$1$-dimensional sheaves are rigid.  A first step is the following.

\begin{lem}
If $M$ is a simple 1-dimensional sheaf on the Poisson surface $Y$,
transverse to the anticanonical curve, then $M$ is rigid iff it is
infinitesimally rigid and $\Fitt_0(M)$ is
an integral curve.
\end{lem}

\begin{proof}
If $M$ has integral support and is infinitesimally rigid, then since
it is simple, we have $\chi_\alpha(M,M)=2$.  It follows that for any
other sheaf $N$ of homological dimension $\le 1$ with the same
numerical Chern class and restriction to $C_\alpha$, we have
\[
\chi_\alpha(M,N)=2.
\]
But this implies either $\Hom(M,N)>0$ or $\Hom(N,M)>0$.  Since $M$
is integral, any subsheaf is either 0 or has $0$-dimensional quotient.
Thus if $f:M\to N$ is a nonzero morphism, then either $f$ is
injective, or $f$ has $0$-dimensional image.  The latter cannot
happen, since $N$ has homological dimension $\le 1$, and the former
implies that $f$ is an isomorphism, by comparison of Chern classes.
Similarly, if $f:N\to M$ is a nonzero morphism, then the cokernel
is $0$-dimensional, but then comparison of Chern classes shows that
the kernel must also be $0$-dimensional, again a contradiction
unless $f$ is an isomorphism.  Either way, we obtain an isomorphism between $M$ and $N$, and thus $M$ is rigid.

Conversely, if $M$ is rigid, then it is certainly infinitesimally rigid
(since symplectic leaves in $\Spl_Y$ are smooth), so only the support needs
to be controlled.  Twisting by a line bundle has no effect on rigidity, so
we may assume that $M$ has no global sections.  If $\Fitt_0(M)$ is not
integral, then we may choose a nonzero subsheaf $N$ with strictly smaller
support, and a point $p$ of the divisor $\Fitt_0(M)-\Fitt_0(N)$ which is
not on $C_\alpha$.  For some sufficiently ample divisor $D$ on $Y$, the
twist $N'=N(D)$ will have a global section.  Define a
sequence of sheaves
\[
M(D)\cong M_{D\cdot c_1(M)}\supset M_{D\cdot
  c_1(M)-1}\supset\cdots\supset M_0,
\]
all containing $N'$, in the following way.  To
obtain $M_{k-1}$ from $M_k$, choose a nonzero map
$M_k/N'\to \sO_p$;
such a map exists since $p$ is in the support of $M_k/N'$.  Then $M_{k-1}$
is the kernel of the induced map $M_k\to \sO_p$.

This process has no effect on the first Chern class, and each step
reduces the Euler characteristic by 1, so that $M_0$ has the same
Chern class as $M$.  In addition, twisting by $\sO(D)$ as no
effect on the restriction to $C_\alpha$ (by transversality), and each
inclusion in the sequence is an isomorphism away from $p$.  So
rigidity implies that $M\cong M_0$.  But $M$ has no global sections,
while $M_0\supset N'$ does.
\end{proof}

\begin{rem}
Note that the argument that a map from $M$ to $N$ must be an
isomorphism also shows that any pure $1$-dimensional sheaf with
integral Fitting scheme is simple.  The argument from constancy of
the index of rigidity is essentially that of \cite[Thm.~1.1.2]{KatzNM:1996}.
\end{rem}

In many cases, the above argument ruling out non-integral support can be
interpreted as twisting by a suitable invertible sheaf on the support, and
noting that the resulting sheaf is nonisomorphic.  We could try to do
something similar in cases with integral support and positive genus; the
main technicality is that the action of invertible sheaves on torsion-free
sheaves has nontrivial stabilizers.  However, if we attempt to identify the
part of $\Ext^1(M,M)$ coming from such twisting, we are led to the
following result.

\begin{lem}
Let $M$ be a pure $1$-dimensional sheaf on $X$ transverse to the
anticanonical curve.  Then there is a natural injection
\[
H^1(\sHom(M,M))\to \im(\Ext^1(M,M\otimes\omega_X)\to \Ext^1(M,M)).
\]
\end{lem}

\begin{proof}
  The local-global spectral sequence for $\Ext^*(M,M)$ collapses at the
  $E_2$ page; indeed, this happens for $\Ext^*(M,N)$ as long as $M$ has
  homological dimension 1 and $N$ has $\le 1$-dimensional support.  We thus
  obtain the following short exact sequence
\[
0\to H^1(\sHom(M,M))\to \Ext^1(M,M)\to H^0(\sExt^1(M,M))\to 0.
\]
Similarly, the spectral sequence for $\Ext^*(M,M\otimes\omega_X)$
collapses, giving a corresponding exact sequence and a commutative diagram
\[
\begin{CD}
H^1(\sHom(M,M)\otimes\omega_X)@>>> \Ext^1(M,M\otimes\omega_X)\\
@VVV @VVV\\
H^1(\sHom(M,M))@>>>\Ext^1(M,M)
\end{CD}
\]
But
\[
H^1(\sHom(M,M)\otimes \sO_{C_\alpha})
=
0,
\]
by dimensionality, and thus the map
\[
H^1(\sHom(M,M)\otimes\omega_X)\to H^1(\sHom(M,M))
\]
is surjective.  The claim follows immediately.
\end{proof}

\begin{rem}
Note that the pairing between $\Ext^1(M,M\otimes\omega_X)$ and
$\Ext^1(M,M)$ restricts to the trivial pairing
\[
H^1(\sHom(M,M)\otimes\omega_X)\otimes H^1(\sHom(M,M))
\to
H^2(\sHom(M,M)\otimes\omega_X)
=
0.
\]
Since the overall pairing is perfect, it follows that
\[
\dim(\im(\Ext^1(M,M\otimes\omega_X)\to \Ext^1(M,M)))
\ge
2h^1(\sHom(M,M)).
\]
This inequality can be strict, e.g., if the first Fitting scheme of $M$ is
$0$-dimensional and not contained in $C_\alpha$.  (Indeed, in that case,
there is another sheaf with the same invariants and smaller endomorphism
sheaf.)
\end{rem}

\begin{thm}
  Let $M$ be a simple, pure $1$-dimensional sheaf on the Poisson surface
  $X$, transverse to the anticanonical curve.  Then $M$ is rigid iff there
  exists a Poisson birational morphism $\pi:Y\to X$ and a $-2$-curve $C$ on
  $Y$ disjoint from the anticanonical curve, such that $M\cong
  \pi_*\sO_C(d)$ for some integer $d$.
\end{thm}

\begin{proof}
We first show that sheaves of the form $\pi_*\sO_C(d)$ are rigid.
Since $C$ is transverse to the exceptional locus, $\sO_C(d)$ is
the minimal lift of its direct image.  Since $\sO_C(d)$ is certainly
simple, and we can compute
\[
\chi_\alpha(\sO_C(d),\sO_C(d)) = -C^2 = 2,
\]
we conclude that it is infinitesimally rigid, and thus so is its direct
image.  Since the image has integral Fitting scheme, the result follows.

  Now, suppose that $M$ is rigid, with (integral) support $C_0$, and let
  $\psi:\tilde{C}_0\to C_0$ be the normalization of $C_0$.  Then we observe
  that
\[
\sHom(M,M)\subset \psi_*\sO_{\tilde{C}_0},
\]
with quotient supported on the singular locus of $C_0$.  Indeed, if $M'$
denotes the quotient of $\psi^*M$ by its torsion subsheaf, then $M'$ is
torsion-free, so invertible, on the curve $\tilde{C}_0$.  This operation is
functorial, and an isomorphism on the smooth locus of $C_0$, so
\[
\sHom(M,M)\subset \psi_*\sHom(M',M')=\psi_*\sO_{\tilde{C}_0}.
\]
Now, $\dim\Hom(M,M)=h^0(\psi_*\sO_{\tilde{C}_0})=1$, and thus we find that
\[
h^1(\sHom(M,M))\ge h^1(\psi_*\sO_{\tilde{C}_0})=h^1(\sO_{\tilde{C}_0}),
\]
with equality only if $\sHom(M,M)=\psi_*\sO_{\tilde{C}_0}$.  Since $M$ is
infinitesimally rigid, the Lemma gives $h^1(\sHom(M,M))=0$, and thus
$\sHom(M,M)=\psi_*\sO_{\tilde{C}_0}$ and $h^1(\sO_{\tilde{C}_0})=0$.

In particular, $\tilde{C}_0$ is a smooth rational curve, and $M$ is the direct
image of a torsion-free sheaf on $\Spec(\sHom(M,M))\cong \tilde{C}_0$.
Since $M$ is the direct image of an invertible sheaf on a smooth curve,
there exists a Poisson birational morphism $\pi:Y\to X$ such that
$\pi^{*!}M$ is disjoint from $\pi^{*!}\sO_{C_\alpha}$.  But then
\[
\chi(\pi^{*!}M)=\chi_\alpha(\pi^{*!}M)=\chi_\alpha(M)=2,
\]
so that $c_1(\pi^{*!}M)^2=-2$, and disjointness gives $c_1(\pi^{*!}M)\cdot
K_Y=0$.  Since the Fitting scheme of $\pi^{*!}M$ is the image of
$\tilde{C}_0$ under a birational morphism, it follows that $\pi^{*!}M$ is
an invertible sheaf on a $-2$-curve as required.
\end{proof}

Note that if $X$ is not rational, then the only possible $-2$ curves on $X$
are components of fibers of the rational ruling of $X$.  Thus rigidity in
Hitchin-type systems is essentially only a phenomenon of the rational case;
we will explore this further in \cite{rat_Hitchin}.

\bibliographystyle{plain}

\end{document}